\documentclass[12pt]{amsart}

\usepackage{CJKutf8}%中文
\usepackage{amssymb}
\usepackage{eucal}
\usepackage{color}
\usepackage{arydshln}
\usepackage{verbatim}

\usepackage{multirow,bigdelim}
\usepackage{amsfonts}
\usepackage{dsfont}
\usepackage{amsmath}
\usepackage{mathrsfs}
\usepackage{todonotes}
\usepackage{hyperref}
\usepackage{multimedia}
\usepackage{graphicx}
\usepackage{indentfirst}
\usepackage{bm}
\usepackage{amsthm}
\usepackage{mathrsfs}
\usepackage{latexsym}
\usepackage{amsmath}
\usepackage{amssymb}
\usepackage{microtype}
\usepackage{relsize}
\usepackage{hyperref}
\usepackage{tikz-cd}
\allowdisplaybreaks[4]

\DeclareSymbolFont{SY}{U}{psy}{m}{n}
\DeclareMathSymbol{\emptyset}{\mathord}{SY}{'306}

\theoremstyle{plain}

\newtheorem{thm}{Theorem}[section]
\newtheorem{cor}[thm]{Corollary}
\newtheorem{lem}[thm]{Lemma}
\newtheorem{prop}[thm]{Proposition}
\theoremstyle{definition}
\newtheorem{defn}[thm]{Definition}
\newtheorem{rem}[thm]{Remark}
\newtheorem{example}[thm]{Example}
\newtheorem{question}[thm]{Question}

\newtheorem*{claim}{Claim}

\newtheorem*{fact}{Fact}
\numberwithin{equation}{section}

\begin{document}

\title{On the quasi-similarity of operators with flag structure}% \uppercase\expandafter{\romannumeral1}}

\author{Yufang Xie, Shanshan Ji, Jing Xu and Kui Ji$^{*}$}

\thanks{*Corresponding author}

\curraddr[Y.F. Xie and K. Ji ]{School of Mathematical Sciences, Hebei Normal University,
Shijiazhuang, Hebei 050024, China}
\curraddr[S.S. Ji]{College of Statistics and Mathematics, Hebei University of Economics and Business, Shijiazhuang, Hebei 050061, China}
\curraddr[J. Xu]{School of Mathematical and Sciences, Hebei GEO University, Shijiazhuang, Hebei 050031, China}

\email[Y.F. Xie]{xieyufangmath@outlook.com}
\email[S.S. Ji]{jishanshan15@outlook.com}
\email[J. Xu]{xujingmath@outlook.com}
\email[K. Ji]{jikui@hebtu.edu.cn, jikuikui@163.com}

\maketitle

\vspace{0.001cm}
\begin{center}
\text{Dedicated to Prof. Gadadhar Misra's 70th birthday}
\end{center}
\vspace{0.001cm}

\begin{abstract}
Let $\mathcal{A}$ denote the operator class in which every nonzero intertwiner between two operators in $\mathcal{A}$ has dense range.
Utilizing the operators in $\mathcal{A}$ as atoms and the flag structure as connection, we introduce an extended operator class $\mathcal{F}_{n}(\mathcal{A}) (n\in\mathbb{N}\ \mbox{and}\ n\ge2)$, along with its subclass $\mathcal{OF}_{n}(\mathcal{A})$.
%We then obtain a clear characterization of the intertwining operator between any two operators with certain properties in $\mathcal{F}_{n}(\mathcal{A})$.
We establish that, under certain conditions, quasi-similarity within the classes $\mathcal{F}_{n}(\mathcal{A})$ and $\mathcal{OF}_{n}(\mathcal{A})$ is equivalent, which provides an approach to describing quasi-similarity and similarity for high-index Fredholm operators.
Furthermore, we demonstrate that quasi-similarity implies similarity for a large number of operators in $\mathcal{F}_{n}(\mathcal{A})$, thereby yielding a partial answer to the question raised by D.A. Herrero in \cite{Herrero} and generalizing existing numerous results.
As applications, several examples of quasi-similarity and similarity involving multiplication operators on vector-valued reproducing kernel Hilbert spaces are presented.
Lastly, we show that the strong irreducibility is preserved up to quasi-similarity within the class $\mathcal{F}_{n}(\mathcal{A})$. This offers a partial solution to C.L. Jiang's question in \cite{JW}.
\end{abstract}

\section{Introduction}
Let $\mathcal{H}$ and $\tilde{\mathcal{H}}$ be two complex separable Hilbert spaces, and let $\mathcal{B}(\mathcal{H},\tilde{\mathcal{H}})$ be the algebra of all bounded linear operators from $\mathcal{H}$ to $\tilde{\mathcal{H}}$.
If $\mathcal{H}=\tilde{\mathcal{H}}$, then denote this algebra by $\mathcal{B}(\mathcal{H})$.
Generally, if a bounded linear operator $X$ is injective and its range is dense, then $X$ is said to be {\bf quasi-affine}.
Given two operators $T\in\mathcal{B}(\mathcal{H})$ and $\tilde{T}\in\mathcal{B}(\tilde{\mathcal{H}})$, if there are two quasi-affine intertwiners $X$ and $Y$ such that $XT=\tilde{T}X$ and $TY=Y\tilde{T}$, then $T$ is {\bf quasi-similar} to $\tilde{T}$, denoted by $T\sim_{q.s}\tilde{T}$.
To study contractive operators on Hilbert spaces, the concept of quasi-similarity was first introduced by B. Sz.-Nagy and C. Foias in their classic book \cite{BF}, in which they proved the existence of two quasi-similar equivalence operators that maintain  superinvariant subspaces.
They also characterized the quasi-similar classification of unitary operators, unilateral shift operators, and operators in the $C_0$ class by studying the structure and properties of these operators.

In \cite{Uchiyama}, M. Uchiyama extended the class of Cowen-Douglas operators $\mathcal{B}_{n}(\Omega)$, which was originally defined by M.J. Cowen and R.G. Douglas in \cite{CD} and exhibit characteristics of complex geometry.
This was achieved using the language of holomorphic curves.
M. Uchiyama also established the quasi-affinity and quasi-similarity classification theorems for contractive operators in this class and the direct sums of unilateral backward shift operators, via tools such as the model theorem and the exact sequence.
Subsequently, M.F. Gamal in \cite{GMF1,GMF2} further generalized the above result on quasi-affinity to power bounded operators, and showed that it does not hold in the quasi-similarity case by constructing counterexamples.
To date, a number of meaningful results concerning the quasi-similarity of operators have been obtained, as detailed in references \cite{Apostol. C, SC, RGD, L.A. Fialkow1, L.A. Fialkow2, HWW, W}.

Recall that two operators $T\in\mathcal{B}(\mathcal{H})$ and $\tilde{T}\in\mathcal{B}(\tilde{\mathcal{H}})$ are said to be \textbf{similar} and are denoted by $T\sim_{s}\tilde{T}$, if there is an invertible operator $X\in\mathcal{B}(\mathcal{H},\tilde{\mathcal{H}})$ satisfying $XT=\tilde{T}X$.
In recent years, research results on similarity equivalence have become increasingly abundant, establishing profound connections between various mathematical branches such as operator theory, algebraic K-theory, reproducing kernel theory and complex geometry, see \cite{JJK, JJM, JJW, JW2, Shields, Uchiyama, XJ}.
Notably, quasi-similarity is equivalent to similarity in finite dimensional spaces.
However, it becomes a substantially weaker relation in infinite dimensional spaces, allowing for differences in the spectrum, essential spectrum, approximate point spectrum, weyl spectrum and compactness between quasi-similar operators.
In spite of these significant findings mentioned above, quasi-similarity remains a subtle relation when consider the following topics:

\textbf{\uppercase\expandafter{\romannumeral1}.\ The existence of two quasi-affine intertwining operators.}\
For a general pair of operators, it's necessary to find two quasi-affine intertwining operators to show their quasi-similarity. Nevertheless, the primary problem in this procedure lies in the boundedness of the two quasi-affine intertwining operators, owing to the fact that densely defined operators may not possess extensions to globally defined bounded linear operators.

\textbf{\uppercase\expandafter{\romannumeral2}.\ Quasi-similarity classification.}
In finite dimensional spaces, the well-known Jordan canonical form theorem\textemdash or, equivalently, similarity invariants\textemdash provides a complete classification of operators up to quasi-similarity.
In infinite dimensional settings, however, the quasi-similarity relation no longer preserves the spectrum of operators (see \cite{Herrero4}), which adds a degree of difficulty to quasi-similarity classification of operators.
%In finite dimensional matrix theory, the Jordan canonical form theorem tells us that each finited imensional matrix is similar to a direct sum of some Jordan blocks.
%The development of the theory of quasi-similarity has also been influenced by the topic of establishing this theorem in infinite-dimensional spaces.
For another,
%there is no universal standard model theory.
%Several scholars have supported using the quasi-similarity relation to study Jordan blocks in infinite-dimensional Hilbert spaces.
%Operators that are quasi-similar to the Jordan operator were found by \cite{ADF,DD,PFV}.
%\textbf{\uppercase\expandafter{\romannumeral3}.\ Quasi-simialrity orbit.}
a full description of the quasi-similarity orbit is known for a rather small set of operators. Results along this line are in \cite{Apostol. C, ADF, DD, Z}.
%In \cite{Apostol. C}, C. Apostol characertized the quasi-simialrity orbit of normal operators. Subsequently, C. Apostol, R.G. Douglas and C. Foias gave a depiction for operators that are quasi-similar to  nilpotent operators in \cite{ADF}. What's more, in \cite{DD}, K.R. Davidson and D.A. Herrero portrayed the quasi-simialrity orbit of bitriangular operators. Recently, W. Zhang showed in \cite{Z} that when an operator is quasi-similar to some complex symmetric operator.

%there is the possibility of losing more properties of the operator.
%Therefore, the quasi-similarity equivalence allows for greater operator flexibility.

\textbf{\uppercase\expandafter{\romannumeral3}.\ Strong irreducibility.}
%Regarding operators on infinite dimensional spaces, a strongly irreducible operator is a suitable substitute for a Jordan block.
Recall that an operator $T\in\mathcal{B}(\mathcal{H})$ is called {\bf strongly irreducible} if $T$ does not commute with any nontrivial idempotent operator, which is denoted by $T\in(SI)$.
%There are also some equivalent definitions of strongly irreducible operators (see \cite{JW}).
%In the 1980s, many scholars conducted extensive research on these operators. Further details can be found in \cite{FJW, Herrero3, JJ, JW3, JW, JW2}.
%It's widely known that the strong irreducibility is invariant under similarity.
In \cite{Apostol. C}, C. Apostol showed that an operator which is quasi-similar to a normal operator must be strongly reducible.
This led many scholars to conjecture that the quasi-similarity transformation no longer preserves the strong irreducibility of the operators.
Indeed, the Example 2.35 in \cite{JW} confirms this conjecture.
Nevertheless, there are still many classes of operators that preserve strong irreducibility under quasi-similarity.
C.L. Jiang and H. He, in \cite{JH}, proved that if $T\in(SI)$ belongs to the Cowen-Douglas class $\mathcal{B}_{n}(\Omega)$ and $\tilde{T}\sim_{q.s}T$, then so is $\tilde{T}$.

Inspired by these literature, this article will mainly consider the following two questions:
\begin{question}[D.A. Herrero, \cite{Herrero}]\label{question1.2}
%Under what conditions on $T\in\mathcal{B}(\mathcal{H})$ does ``$A\sim_{q.s}T$" imply ``$A\sim_{s}T$"?
Under what conditions on $T\in\mathcal{B}(\mathfrak{X})$ does ``$A\sim_{q.s}T$" imply ``$A\sim_{s}T$", where $\mathfrak{X}$ is a complex Banach space?
\end{question}

\begin{question}[C.L. Jiang, \cite{JW}]\label{question1.1}
What kinds of strongly irreducible operators preserve strong irreducibility under quasi-similarity?
\end{question}

Question \ref{question1.2} addresses the conditions under which a quasi-similarity relation between operators can be strengthened to a similarity relation.
This problem has been investigated in various contexts.
Notably, A.L. Lambert demonstrated in \cite{A.L. Lambert} (see also \cite{Shields}) that two injective unilateral shift operators that are quasi-similar must in fact be similar.
Moreover, L.R. Williams established in \cite{Williams} that the same is true for quasi-similar hyponormal (possibly noninjective) unilateral weighted shifts.
Additionally, L.A. Fialkow proved in \cite{L.A. Fialkow1} that for two bilateral weighted shift operators, quasi-similarity combined with the invertibility of one operator guarantees their similarity.
In \cite{Uchiyama}, M. Uchiyama showed that any two quasi-similar homogeneous operators in the Cowen-Douglas class with index one are necessarily similar.

%C.L. Jiang and Z.Y. Wang proved in their monograph \cite{JW} that each operator in the Cowen-Douglas operator class $\mathcal{B}_n(\Omega)$ can be represented as an $n\times n$ matrix in upper-triangle form.
%Building on this structural theorem, the fourth author and his collaborators introduced the flag structure into the study of such operators in \cite{JJKM}.
%We note that $\mathcal{FB}_n(\Omega)$ is the set of operators with the flag structure in the Cowen-Douglas operator class, $n\geq2$.
%This class contains a large number of operators, including high-index irreducible homogeneous operators in \cite{JJM}.
%Moreover, the completely unitary invariants of operators in this class only involve the curvature and the second fundamental form of the line bundle corresponding to the diagonal operator.
%In \cite{XJ}, the first and fourth authors further defined the strong flag structure, and operators with this structure are norm-dense in $\mathcal{B}_n(\Omega)$ up to similarity.
%We denote all such operators in $\mathcal{B}_n(\Omega)$ as $\mathcal{OFB}_n(\Omega)$.
%The similarity equivalence problem of a large number of operators in $\mathcal{FB}_n(\Omega)$ can be transformed into the similarity equivalence problem for operators in $\mathcal{OFB}_n(\Omega)$.

To further investigate the aforementioned two open problems, the present work introduces a new subclass $\mathcal{A}$ of bounded linear operators exhibiting wider applicability.
The Cowen-Douglas class $\mathcal{B}_{1}(\Omega)$ serves as an example of class $\mathcal{A}$.
However, $\mathcal{A}$ is not limited to such operators and may in fact contain numerous operators that lie outside the Cowen-Douglas class with index one.
Then, just as molecules are composed of atoms, our aim is to obtain operators in the form of an upper-triangular operator matrix that can be regarded as ``molecules", with the operators in $\mathcal{A}$  being the ``atoms" and the flag structure being the ``bonding". %The flag structure was introduced by the fourth author and his collaborators in \cite{JJKM}.
Let $\mathcal{F}_{n}(\mathcal{A})$ denote the class of operators obtained in this way.
%It's easy to know that the class $\mathcal{F}_{n}(\mathcal{A})$ is exactly the class $\mathcal{FB}_n(\Omega)$, a subclass of $\mathcal{B}_{n}(\Omega)$ and introduced by the fourth author and his collaborators in \cite{JJKM}, when consider $\mathcal{A}$ as $\mathcal{B}_1(\Omega))$.
We will now consider Question \ref{question1.2} and Question \ref{question1.1} concerning this new and large operator class, and provide sufficient conditions for their validity.
Of particular interest is a newly defined subclass $\mathcal{OF}_{n}(\mathcal{A})$ exhibiting strong flag structure.
As Definition \ref{definition4.2} indicates, this subclass admits a clear and simple structure.
It is of fundamental importance as it serves as a bridge that effectively transforms quasi-similarity relations between the classes $\mathcal{F}_{n}(\mathcal{A})$ and $\mathcal{OF}_{n}(\mathcal{A})$, yielding remarkable reductions in complexity.

\begin{comment}
As shown, the classification problem of operators are closely related to the property of their intertwining operators, which has always been a hot topic in the research of operator theory. In \cite{Uchiyama},  M. Uchiyama showed a characterization for contractions in the Cowen-Douglas class which are quasi-similar to the adjoint of the multiplication operators $S_{n}^{*}$ on the Hardy space with finite multiplicity $n$.
\begin{lem}[\cite{Uchiyama}, Theorem 3.3]
Let $T\in\mathcal{B}_{n}(\mathbb{D})$ and $\Vert T\Vert\le1$. Then $T\prec S_{n}^{*}$ if and only if there exists a frame $\{\gamma_{1},\cdots,\gamma_{n}\}$ for the Herimitian holomorphic vectore bundle $\mathcal{E}_{T}$ such that
\begin{equation*}
\setlength\abovedisplayskip{3pt}
\setlength\belowdisplayskip{3pt}
\sup\limits_{w\in\mathbb{D}}(1-\vert w\vert^{2})\Vert\gamma_{i}(w)\Vert^{2}<\infty,\ \forall\ 1\le i\le n.
\end{equation*}
\end{lem}
\end{comment}

The article is structured as follows.
In section \ref{sec2}, the fundamental concepts and conclusions are outlined.
In section \ref{sec3}, we introduce a new class $\mathcal{A}$ of bounded linear operators. Then in the language of operator matrices and incorporating flag structure, we define a class $\mathcal{F}_{n}(\mathcal{A})$ of high-index Fredholm operators. Subsequently, we investigate the structure of the intertwining operator for any pair of operators within this class (specific details can be found in the Appendix). More importantly, we identify a subclass $\mathcal{OF}_{n}(\mathcal{A})$ with strong flag structure. We further establish that the quasi-similarity relationship among operators in this subclass is equivalent to that of operators in $\mathcal{F}_{n}(\mathcal{A})$.
In section \ref{sec4}, %a close relationship is identified between similarity for operators derived in the previous section, satisfying certain conditions, and similarity for operators with simpler structures.
we demonstrate that a class of operators can be obtained for which similarity and quasi-similarity are equivalent, thereby providing a partial answer to Question \ref{question1.2} posed by D.A. Herrero.
Since the class $\mathcal{F}_{n}(\mathcal{A})$ contains many adjoints of multiplication operators on reproducing kernel Hilbert spaces, we also discuss their similarity and quasi-similarity.
In section \ref{sec5}, we show that the high-index operators introduced in Section \ref{sec3} preserve the strong irreducibility  under quasi-similarity.
%All of these operators are generated by operators in this new class. %and possess a flag structure.
This result offers a partial response to Question \ref{question1.1} posed by C.L. Jiang.

\section{preliminaries on weighted shift operators}\label{sec2}

A unilateral weighted shift operator $T$ on a Hilbert space $\mathcal{H}$ maps each vector in an orthonormal basis $\{e_{n}\}_{n\ge0}$ to a scalar multiple of the next vector: $Te_{n}=w_{n}e_{n+1}$ for all $n\ge0$, where $\{w_{n}\}_{n\ge0}$ is the weight sequence.
If none of the weights are zero, then the weighted shift operator $T$ is injective.
In \cite{Shields}, another way of viewing injective unilateral weighted shift operators is to represent them as the multiplication operator $M_{z}$ on a Hilbert space of formal power series, $\mbox{H}^{2}(\beta)$, where
$$\mbox{H}^{2}(\beta)=\left\{f(z)=\sum\limits_{n=0}^{\infty}\hat{f}(n)z^{n}:\Vert f\Vert^{2}=\sum\limits_{n=0}^{\infty}\vert\hat{f}(n)\vert^{2}(\beta(n))^{2}\right\}$$
for $\beta(0)=1,\ \beta(n)=\prod\limits_{k=0}^{n-1}w_{k},\ n>0$.
Denoted by  $$\mbox{H}^{\infty}(\beta)=\left\{\phi(z)=\sum\limits_{n=0}^{\infty}\hat{\phi}(n)z^{n}:\phi\mbox{H}^{2}(\beta)\subset\mbox{H}^{2}(\beta)\right\}.$$

For the sake of convenience, the results we need are listed here:

\begin{lem}[\cite{Shields}, Theorem 3]\label{Lemma2.1}
If $A$ is an operator on $\mbox{H}^{2}(\beta)$ that commutes with $M_{z}$, then $A=M_{\phi}$ for some $\phi\in\mbox{H}^{\infty}(\beta)$.
\end{lem}

\begin{lem}[\cite{Shields}, Theorem 10]\label{lemma2.2}
Let $T$ be an injective unilateral weighted shift operator with weight sequence $\{w_{n}\}_{n\ge0}$, represented as $M_{z}$ on the space $\mbox{H}^{2}(\beta)$.
If $\vert w\vert<r(T)$ and if $\phi\in\mbox{H}^{\infty}(\beta)$, where $r(T)$ is the radius of the spectrum of $T$, then the power series of $\phi$ converges at $w$ and $\vert\phi(w)\vert\le\Vert M_{\phi}\Vert$.
\end{lem}

For $T\in\mathcal{B}(\mathcal{H})$, the \textbf{commutant\ algebra} of $T$ is a weakly closed subalgebra with identity of $\mathcal{B}(\mathcal{H})$,
\begin{equation*}
\setlength\abovedisplayskip{3pt}
\setlength\belowdisplayskip{3pt}
\{T\}':=\{A\in\mathcal{B}(\mathcal{H}):TA=AT\}=\mbox{the\ commutant\ algebra\ of\ T}.
\end{equation*}
We denote the Jacobson radical of $\{T\}'$ by $rad\ \{T\}'$, and say that $\{T\}'$ is \textbf{semi-simple} if $rad\ \{T\}'=\{0\}$.
In this case, a basic fact is that there are no nonzero quasi-nilpotent operators in $\{T\}'$, where an operator $P\in\mathcal{B}(\mathcal{H})$ is said to be \textbf{quasi-nilpotent} if $\lim\limits_{n\rightarrow\infty}\Vert P^{n}\Vert^{\frac{1}{n}}=0$.

\begin{prop}\label{250721}
Let $T$ be an injective unilateral weighted shift operator with weight sequence $\{w_{n}\}_{n\ge0}$ satisfying $r(T)\ne0$, represented as $M_{z}$ on the space $\mbox{H}^{2}(\beta)$.
In this case, the commutant algebra of $T$ is semi-simple.
\end{prop}
\begin{proof}
Suppose that $X$ is a quasi-nilpotent operator in $\{T\}'$. By Lemma \ref{Lemma2.1}, $X$ admits a representation $X=M_{\phi}$ for some $\phi\in\mbox{H}^{\infty}(\beta)$.
Since $r(T)\ne0$, there exists an open set $U\subset\sigma(T)$ that is symmetric about the origin.
Applying Lemma \ref{lemma2.2}, we deduce that the function $\phi$ is analytic on $\sigma(T)$.
Consequently, $\phi(U)\subset\phi(\sigma(T))=\sigma(M_{\phi})=\{0\}$, which implies that $\phi\equiv0$. It follows that $X=0$, proving that $\{T\}'$ is semi-simple. This completes the proof.
\end{proof}

For $A,B\in\mathcal{B}(\mathcal{H})$, the \textbf{Rosenblum operator} $\tau_{A,B}:\mathcal{B}(\mathcal{H})\rightarrow\mathcal{B}(\mathcal{H})$ is defined as
\begin{equation*}
\setlength\abovedisplayskip{3pt}
\setlength\belowdisplayskip{3pt}
\tau_{A,B}(X):=AX-XB,\ X\in\mathcal{B}(\mathcal{H}).
\end{equation*}
If $A=B$, write it as $\tau_{A}$.

\begin{lem}[\cite{JJK}, Proposition 3.5]\label{lemma4.5}
Let $A$ and $B$ be two backward weighted shift operators acting on the Hilbert space $\mathcal{H}$ with the weight sequences $\{a_{i}\}_{i\ge0}$ and $\{b_{i}\}_{i\ge0}$, respectively. If $X$ intertwines $A$ and $B$, i.e. $AX=XB$, then there exists an orthonormal basis $\{e_{i}\}_{i\ge0}$ of $\mathcal{H}$ such that $X$ has matrix representation $X=\big((x_{i,j})\big)_{i,j\ge0}$ with respect to this basis,
%\begin{equation*}
%\setlength\abovedisplayskip{3pt}
%\setlength\belowdisplayskip{3pt}
%X=\left( \begin{smallmatrix}
%		x_{0,0}&x_{0,1}&x_{0,2}&\cdots&x_{0,n}&\cdots\\
%		&x_{1,1}&x_{1,2}&\cdots&x_{1,n}&\cdots\\
%		&&\ddots&\ddots&\vdots&\cdots\\
%		&&&x_{n-1,n-1}&x_{n-1,n}&\cdots\\
%		&&&&x_{n,n}&\ddots\\
%		&&&&&\ddots
%	\end{smallmatrix}\right),
%\end{equation*}
where $x_{i,j}=0$ if $i>j$ and $x_{n+1,n+j}=\frac{\prod\limits_{k=j-1}^{n+j-1}b_{k}}{\prod\limits_{k=0}^{n}a_{k}}x_{0,j-1}$, $j=1,2,\cdots$, $n=0,1,\cdots$. Moreover, if $\lim\limits_{n\rightarrow\infty}n\frac{\prod\limits_{k=0}^{n-1}b_{k}}{\prod\limits_{k=0}^{n-1}a_{k}}=\infty$, then $\ker\tau_{A,B}\cap\mbox{ran}\ \tau_{A,B}=\{0\}$.
\end{lem}

\begin{lem}\cite{Shields}\label{lemma2.5}
If two injective unilateral shifts are quasi-similar, then they are similar.
\end{lem}

\section{A new class of operators}\label{sec3}
In this section, we will introduce a new class $\mathcal{A}$ of bounded linear operators acting on a Hilbert space.
Building upon the operators in $\mathcal{A}$ and $\{X\in\ker\tau_{T,\tilde{T}}-\{0\}:T,\tilde{T}\in\mathcal{A}\}$, we further define a high-index operator class $\mathcal{F}_{n}(\mathcal{A})$, each exhibiting a flag structure. Additionally, we identify a structurally well-behaved subclass $\mathcal{OF}_{n}(\mathcal{A})\subset\mathcal{F}_{n}(\mathcal{A})$, characterized by a clear and simplified form. We then explore the similarity relationship between $\mathcal{F}_{n}(\mathcal{A})$ and $\mathcal{OF}_{n}(\mathcal{A})$.
%Thus, the intertwining relation of any two operators in $\mathcal{F}_{n}(\mathcal{A})$ shows the intertwining relation between the corresponding atoms.
This result enables the reduction of the quasi-similarity problem in a general class $\mathcal{F}_{n}(\mathcal{A})$ to a corresponding problem in a more structured subclass $\mathcal{OF}_{n}(\mathcal{A})$, thus establishing a method to address the equivalence between quasi-similarity and similarity for high-index Fredholm operators.

Firstly, we will explain the meaning of symbol $\mathcal{A}$.

\begin{defn}\label{definition3.1}
Let $\mathcal{A}$ be an operator class such that all operators in it are bounded and
\begin{equation}\label{equation3.1}
\setlength\abovedisplayskip{3pt}
\setlength\belowdisplayskip{3pt}
\mbox{for\ any}\ T\ \mbox{and}\ \tilde{T}\ \mbox{in}\ \mathcal{A},\ X\in\ker\tau_{T,\tilde{T}}-\{0\}\ \mbox{has\ dense\ range}.
\end{equation}
\end{defn}

As established by Proposition 2.3 in \cite{JJKM}, the Cowen-Douglas class $\mathcal{B}_{1}(\Omega)$ possesses property (\ref{equation3.1}).
We now demonstrate that this property can extend to the class of backward weighted shift operators. Specifically, we prove:
%Furthermore, the following Proposition shows the backward weighted shift operators class also satisfies property (\ref{equation3.1}).

\begin{prop}\label{proposition3.2}
Let $A$ and $B$ be two backward weighted shift operators with nonzero weighted sequences $\{a_{i}\}_{i\ge0}$ and $\{b_{i}\}_{i\ge0}$, respectively. Suppose that the bounded linear operator $X$ intertwines $A$ and $B$, i.e. $XA=BX$. Then $X$ is nonzero if and only if $X^{*}$ is injective, which is equivalent to saying that $X$ has dense range.
\end{prop}
\begin{proof}
Suppose that the operators $A$, $B$ and $X$ act on a Hilbert space $\mathcal{H}$, and that the set of orthonormal basis which enable $A$ and $B$ to be shift operators is $\{e_{i}\}_{i\ge0}$.
Let $X^{*}=\big((x_{i,j})\big)_{i,j\ge0}$ be the matrix representation of the operator $X^{*}$ with respect to $\{e_{i}\}_{i\ge0}$. Following the relation $XA=BX$, we obtain that $A^{*}X^{*}=X^{*}B^{*}$ and $X^{*}$ is in lower-triangular form, where
\begin{equation}\label{equation3.2}
\setlength\abovedisplayskip{3pt}
\setlength\belowdisplayskip{3pt}
x_{n,n}=\frac{\prod\limits_{k=0}^{n-1}a_{k}}{\prod\limits_{k=0}^{n-1}b_{k}}x_{0,0},\ x_{n+j,n}=\frac{\prod\limits_{k=j}^{n}a_{k}}{\prod\limits_{k=0}^{n-1}b_{k}}x_{j,0},\ n,j=1,2,\cdots.
\end{equation}
It's obvious that the injectivity of $X^{*}$ will imply that $X$ is nonzero.
Conversely, if $X$ is nonzero, which is equivalent to $X^{*}$ being nonzero, then $X^{*}(e_{0})$ must be nonzero.
Subsequently, pick $q$ to be the smallest indices that $x_{q,0}\ne0$ and assume that $x=\sum\limits_{k=0}^{\infty}\langle x,e_{k}\rangle e_{k}\in\ker X^{*}$.
It follows that
\begin{equation*}
\setlength\abovedisplayskip{3pt}
\setlength\belowdisplayskip{3pt}
0=X^{*}(\sum\limits_{k=0}^{\infty}\langle x,e_{k}\rangle e_{k})
=\sum\limits_{k=0}^{\infty}\langle x,e_{k}\rangle X^{*}(e_{k})
=\sum\limits_{k=0}^{\infty}\langle x,e_{k}\rangle(\sum\limits_{i=k}^{\infty}x_{i,k}e_{i})
=\sum\limits_{n=0}^{\infty}\big(\sum\limits_{k=0}^{n}\langle x,e_{k}\rangle x_{n,k}\big)e_{n},
\end{equation*}
which shows that $\sum\limits_{k=0}^{n}\langle x,e_{k}\rangle x_{n,k}=0$ for all $n\ge0$.
Since $x_{k,0}=0$ for $k=0,\cdots,q$, and equation (\ref{equation3.2}), we have the equality $x_{q+m,l}=0$ for $l\ge m+1$ and $m\ge0$.
Therefore, using the fact that $\sum\limits_{k=0}^{q}\langle x,e_{k}\rangle x_{q,k}=0$ and $x_{q,0}\ne0$, we obtain that $\langle x,e_{0}\rangle=0$.
Furthermore, from
\begin{equation*}
\setlength\abovedisplayskip{3pt}
\setlength\belowdisplayskip{3pt}
0=\sum\limits_{k=0}^{q+1}\langle x,e_{k}\rangle x_{q+1,k}=\sum\limits_{k=0}^{1}\langle x,e_{k}\rangle x_{q+1,k}=\langle x,e_{1}\rangle x_{q+1,1},
\end{equation*}
one can see that $\langle x,e_{1}\rangle=0$.
Now suppose that we have proved $\langle x,e_{n}\rangle=0$ for $n=0,\cdots,k$.
Then together with $x_{q+k+1,k+1}\ne0$ and
\begin{equation*}
\setlength\abovedisplayskip{3pt}
\setlength\belowdisplayskip{3pt}
0=\sum\limits_{n=0}^{q+k+1}\langle x,e_{n}\rangle x_{q+k+1,n}=\sum\limits_{n=k+1}^{q+k+1}\langle x,e_{n}\rangle x_{q+k+1,n}=\langle x,e_{k+1}\rangle x_{q+k+1,k+1},
\end{equation*}
we have $\langle x,e_{k+1}\rangle=0$. Thus $x=0$, which implies that $X^{*}$ is injective.
\end{proof}

\begin{rem}
Since a backward weighted shift operator (or in short, an \textbf{ONB shift}) is not necessarily a Cowen-Douglas operator with index one, the operator class $\mathcal{A}$ introduced above strictly encompasses the Cowen-Douglas class $\mathcal{B}_{1}(\Omega)$. Moreover, we've discovered a class of operators\textemdash\textbf{M-basis shifts} (cf. \cite{GS2})\textemdash which is more general than the Cowen-Douglas class $\mathcal{B}_{1}(\Omega)$ and the class of ONB shifts, while satisfying property (\ref{equation3.1}).
Specific details will appear in a forthcoming work, and we have also investigated other characteristics of these shifts.
\end{rem}

We proceed to utilize the flag structure as bonding to create a new class in the language of upper-triangular operator matrices. What'more, we think that condition (3) in the following definition corresponds to the \textbf{flag structure} of the operator $T$.

\begin{defn}\label{definition3.3}
We say that an operator $T\in\mathcal{B}(\mathcal{H})$ is in the class \boldsymbol{$\mathcal{F}_{n}(\mathcal{A})(n\ge2)$} if $T$ satisfies the following conditions:
\begin{enumerate}
\item $T$ can be written as an $n\times n$ upper-triangular matrix in the form $T=\big((T_{i,j})\big)_{i,j=1}^{n}$ with respect to a topological direct decomposition of $\mathcal{H}=\mathcal{H}_{1}\oplus\cdots\oplus\mathcal{H}_{n}$;

\item $T_{i,i}\in\mathcal{A}$ for $1\le i\le n$;% and $\{T_{i,i}\}'$ is semi-simple commutative Banach algebra;

\item $T_{i,i+1}\ne 0$ and $T_{i,i}T_{i,i+1}=T_{i,i+1}T_{i+1,i+1}$ for $1\le i\le n-1$.

\end{enumerate}
\end{defn}

Moreover, Definition \ref{definition4.2} reveals the existence of a subclass characterized by operator matrices whose nonzero entries are on the main diagonal and the first subdiagonal.
We refer to this specific structure as a \textbf{strong flag structure}. The main result of this section, as the following Theorem \ref{lemma4} shown, is that the quasi-similarity of operators in $\mathcal{F}_{n}(\mathcal{A})$  can be transformed into the quasi-similarity of operators in $\mathcal{OF}_{n}(\mathcal{A})$ under Condition (A). This conclusion will play a central role in the study of the equivalence problem between quasi-similarity and similarity of high-index Fredholm operators in the next section.
%The following definitions and conclusions generalizes the Definitions 3.2 and 3.6 and Lemma 3.9 in \cite{XJ}.

\begin{defn}\label{definition4.2}
	An operator $T=\big((T_{i,j})\big)_{i,j=1}^{n}\in\mathcal{F}_{n}(\mathcal{A})$ is in the class \boldsymbol{$\mathcal{OF}_{n}(\mathcal{A})$} if $T_{i,j}=0$ for $1\le i<i+2\le j\le n$.
\end{defn}

\begin{defn}\label{definition4.3}
	An operator $T=\big((T_{i,j})\big)_{i,j=1}^{n}\in\mathcal{F}_{n}(\mathcal{A})$ is said to satisfy \textbf{Condition (A)} if $T_{i,j}=\phi_{i,j}(T_{i,i})T_{i,i+1}T_{i+1,i+2}\cdots T_{j-1,j}$, where $\phi_{i,j}(T_{i,i})\in\{T_{i,i}\}'$ for $1\le i<i+2\le j\le n$.
\end{defn}

The main result of this section is as follows.
\begin{thm}\label{lemma4}
	Let $T=\big((T_{i,j})\big)_{i,j=1}^{n},\tilde{T}=\big((\tilde{T}_{i,j})\big)_{i,j=1}^{n}\in\mathcal{F}_{n}(\mathcal{A})$, where $\{T_{i,i}\}'$ (resp., $\{\tilde{T}_{i,i}\}'$) is a semi-simple commutative Banach algebra for $1\le i\le n$ and both $T$ and $\tilde{T}$ satisfy Condition (A). Then $T\sim_{q.s}\tilde{T}$ if and only if $S\sim_{q.s}\tilde{S}$, where $S=\big((S_{i,j})\big)_{i,j=1}^{n},\tilde{S}=\big((\tilde{S}_{i,j})\big)_{i,j=1}^{n}\in\mathcal{OF}_{n}(\mathcal{A})$ and $S_{i,j}=T_{i,j}$, $\tilde{S}_{i,j}=\tilde{T}_{i,j}$ for $i=j$ and $i+1=j$.
\end{thm}

Prior to proving this main theorem, we need to provide some necessary explanations.
One advantage of operator class $\mathcal{F}_{n}(\mathcal{A})$, defined by the language of operator matrices and flag structure, is that every operator in the commutant algebra of an operator in $\mathcal{F}_{n}(\mathcal{A})$ admits a nice property.
More precisely, every operator intertwining an operator in class $\mathcal{F}_{n}(\mathcal{A})$ has an operator matrix in the form of upper-triangular.
Furthermore, when we consider the quasi-similarity topic, we describe that any two quasi-affine intertwiners between two different operators in $\mathcal{F}_{n}(\mathcal{A})$ are also upper-triangular form, which is a fundamental but crucial prerequisite work to prove the main theorem.
%Notably, when $\mathcal{A}=\mathcal{B}_{1}(\Omega)$, the class $\mathcal{F}_{n}(\mathcal{A})$ coincides with the class $\mathcal{FB}_{n}(\Omega)$.
%This operator class has been extensively studied in \cite{JJKM,JJK,JJW,JJM} in relation to topics such as unitary classification, similarity classification, irreducibility and strong irreducibility.
%An essential role is played by the operator's commutant algebra in these topics.
%In fact, we describe that any intertwining operator in the commutant algebra of the operator in $\mathcal{F}_{n}(\mathcal{A})$ is upper-triangular form.
The reader can be referred to the Appendix for the related conclusions.

The following Proposition establishes the bridge between the classes $\mathcal{F}_{n}(\mathcal{A})$ and $\mathcal{OF}_{n}(\mathcal{A})$, which enables to transfer some problems on the class $\mathcal{F}_{n}(\mathcal{A})$ to the corresponding problems on the class $\mathcal{OF}_{n}(\mathcal{A})$.

\begin{prop}\label{lemma4.4}
Let $T=\big((T_{i,j})\big)_{i,j=1}^{n}\in\mathcal{OF}_{n}(\mathcal{A})$ and $S=\big((S_{i,j})\big)_{i,j=1}^{n}\in\mathcal{F}_{n}(\mathcal{A})$,
where $T_{i,j}=S_{i,j}$ for $i=j$ and $i+1=j$, $\{T_{i,i}\}'$ is a semi-simple commutative Banach algebra for $1\le i\le n$, respectively, and $S$ satisfies Condition (A).
Then there exists a bounded operator $K$ such that $X:=I+K$ is invertible and $XT=SX$.
\end{prop}
\begin{proof}
	We first prove the case of $n=3$.  Following from the Proposition \ref{proposition3.7} we can set $X=\left(\begin{smallmatrix}
		I&K_{1,2}&K_{1,3}\\
		0&I&K_{2,3}\\
		0&0&I
	\end{smallmatrix}\right)$, then we consider the following equation:
\begin{equation}
\setlength\abovedisplayskip{3pt}
\setlength\belowdisplayskip{3pt}
	\left(\begin{smallmatrix}
		I&K_{1,2}&K_{1,3}\\
		0&I&K_{2,3}\\
		0&0&I
	\end{smallmatrix}\right)
	\left(\begin{smallmatrix}
	T_{1,1}&T_{1,2}&0\\
	0&T_{2,2}&T_{2,3}\\
	0&0&T_{3,3}
    \end{smallmatrix}\right)
=
\left(\begin{smallmatrix}
	T_{1,1}&T_{1,2}&S_{1,3}\\
	0&T_{2,2}&T_{2,3}\\
	0&0&T_{3,3}
\end{smallmatrix}\right)
	\left(\begin{smallmatrix}
	I&K_{1,2}&K_{1,3}\\
	0&I&K_{2,3}\\
	0&0&I
\end{smallmatrix}\right).
\end{equation}
This implies that  for $i=1,2$,
\begin{equation*}
\setlength\abovedisplayskip{3pt}
\setlength\belowdisplayskip{3pt}
\begin{aligned}
T_{i,i}K_{i,i+1}&=K_{i,i+1}T_{i+1,i+1},\ i=1,2;\\ T_{1,1}K_{1,3}-K_{1,3}T_{3,3}&=K_{1,2}T_{2,3}-T_{1,2}K_{2,3}-S_{1,3}.
\end{aligned}
\end{equation*}
Since $S$ satisfies Condition (A), we have  $S_{1,3}=\phi_{1,3}(T_{1,1})T_{1,2}T_{2,3}$, where $\phi_{1,3}(T_{1,1})\in\{T_{1,1}\}'$.

Let $K_{1,3}=S_{1,3}$ and $K_{2,3}=T_{2,3}$, then we have that $K_{1,3}T_{3,3}=T_{1,1}K_{1,3}$ and $T_{2,2}K_{2,3}=K_{2,3}T_{3,3}$.
Furthermore, we set $K_{1,2}=T_{1,2}+\phi(T_{1,1})T_{1,2}$, which satisfies that $T_{1,1}K_{1,2}=K_{1,2}T_{2,2}$.
Thus the operator $X:=\left(\begin{smallmatrix}
I&T_{1,2}+\phi(T_{1,1})T_{1,2}&T_{1,3}\\
0&I&T_{2,3}\\
0&0&I
\end{smallmatrix}\right)$ is invertible and satisfies that $XT=SX$ by a routine calculation.

The proof is by induction for $n$.
Suppose that the conclusion holds for all $k<n$.
To make this more precise, for all $k<n$, we assume that there exists a bounded operator $K=\big((K_{i,j})\big)_{i,j=1}^{k}$, where $K_{i,j}=0$ for $i\ge j$, $K_{i,j}=\psi_{i,j}(T_{i,i})T_{i,i+1}\cdots T_{j-1,j}$ for $1\le i<j<k$ and $K_{i,k}=S_{i,k}$, such that $X:=I+K$ is invertible and $XT=SX$.

Next we consider the general case of $n$.
Let us write the two operators $T,S$ in the form of $2\times 2$ block matrices:
\begin{equation*}
	\setlength\abovedisplayskip{3pt}
	\setlength\belowdisplayskip{3pt}
	T=\left(\begin{smallmatrix}T_{1,1}&T_{1\times(n-1)}\\0&T_{n-1}\end{smallmatrix}\right),\
	S=\left(\begin{smallmatrix}T_{1,1}&S_{1\times(n-1)}\\0&S_{n-1}\end{smallmatrix}\right),
\end{equation*}
where $T_{n-1}=\big((T_{i,j})\big)_{i,j=2}^{n}\in\mathcal{OF}_{n-1}(\mathcal{A})$, $T_{1\times(n-1)}=(T_{1,2},0,\cdots,0)$,  $S_{n-1}=\big((S_{i,j})\big)_{i,j=2}^{n}\in\mathcal{F}_{n-1}(\mathcal{A})$ and $S_{1\times(n-1)}=(T_{1,2},S_{1,3},\cdots,S_{1,n})$.
Correspondingly, $X=I+K$ can be written as
$X=\left(\begin{smallmatrix}I&X_{1\times(n-1)}\\0&X_{n-1}\end{smallmatrix}\right)$,
where $X_{n-1}=I+K_{n-1}$ and $X_{1\times(n-1)}=(K_{1,2},K_{1,3},\cdots,K_{1,n})$.
From the relation $XT=SX$, we get that
\begin{equation*}
	\setlength\abovedisplayskip{3pt}
	\setlength\belowdisplayskip{3pt}
    \left(\begin{smallmatrix}I&X_{1\times(n-1)}\\0&X_{n-1}\end{smallmatrix}\right)
	\left(\begin{smallmatrix}T_{1,1}&T_{1\times(n-1)}\\0&T_{n-1}\end{smallmatrix}\right),\
	=\left(\begin{smallmatrix}T_{1,1}&S_{1\times(n-1)}\\0&S_{n-1}\end{smallmatrix}\right)
	\left(\begin{smallmatrix}I&X_{1\times(n-1)}\\0&X_{n-1}\end{smallmatrix}\right),
\end{equation*}
which shows that $$T_{1\times(n-1)}+X_{1\times(n-1)}T_{n-1}=T_{1,1}X_{1\times(n-1)}+S_{1\times(n-1)}X_{n-1}.$$
That is to say that we have, for $2\le i\le n$,
\begin{equation}\label{equation3.5}
	\setlength\abovedisplayskip{3pt}
	\setlength\belowdisplayskip{3pt}
	K_{1,i-1}T_{i-1,i}+K_{1,i}T_{i,i}=T_{1,1}K_{1,i}+T_{1,2}K_{2,i}+S_{1,3}K_{3,i}+
	+\cdots+S_{1,i-1}K_{i-1,i}+S_{1,i}.
\end{equation}
Let $K_{1,n}=S_{1,n}$, then we can know that $K_{1,n}T_{n,n}=T_{1,1}K_{1,n}$. Thus together the equation obtained by setting $i=n$ in Equation (\ref{equation3.5}) with the hypothesis, we get that
\begin{equation*}
	\setlength\abovedisplayskip{3pt}
	\setlength\belowdisplayskip{3pt}
	\begin{aligned}
    K_{1,n-1}&=T_{1,2}\phi_{2,n}(T_{2,2})T_{2,3}\cdots T_{n-2,n-1}+\cdots+\phi_{1,n}(T_{1,1})T_{1,2}\cdots T_{n-2,n-1}\\
    &\triangleq\psi_{1,n-1}(T_{1,1})T_{1,2}\cdots T_{n-2,n-1},
    \end{aligned}
\end{equation*}
which satisfies that $K_{1,n-1}T_{n-1,n-1}=T_{1,1}K_{1,n-1}$. Moreover, together the equation obtained by setting $i=n-1$ with the $K_{1,n-1}$, we know that
\begin{equation*}
	\setlength\abovedisplayskip{3pt}
	\setlength\belowdisplayskip{3pt}
	\begin{aligned}
    K_{1,n-2}&=T_{1,2}\psi_{2,n-1}(T_{2,2})T_{2,3}\cdots T_{n-3,n-2}+\cdots+\phi_{1,n-1}(T_{1,1})T_{1,2}\cdots T_{n-3,n-2}\\
    &\triangleq\psi_{1,n-2}(T_{1,1})T_{1,2}\cdots T_{n-3,n-2},
    \end{aligned}
\end{equation*}
which satisfies that $K_{1,n-2}T_{n-2,n-2}=T_{1,1}K_{1,n-2}$. By repeating the previous step, suppose that we have solved $K_{1,i}$ satisfying $K_{1,m}T_{m,m}=T_{m,m}K_{1,m}$ for $3\le m\le n$. Now together the equation obtained by setting $i=m$ in Equation (\ref{equation3.5}) with the $K_{1,m}$, we can derive that
\begin{equation*}
 	\setlength\abovedisplayskip{3pt}
 	\setlength\belowdisplayskip{3pt}
 	\begin{aligned}
 	K_{1,m-1}&=T_{1,2}\psi_{2,m}(T_{2,2})T_{2,3}\cdots T_{m-2,m-1}+\cdots+\phi_{1,m}(T_{1,1})T_{1,2}\cdots T_{m-2,m-1}\\
 	&\triangleq\psi_{1,m-1}(T_{1,1})T_{1,2}\cdots T_{m-2,m-1}.
 	\end{aligned}
\end{equation*}

In summary, we find a bounded operator $K=\big((K_{i,j})\big)_{i,j=1}^{n}$, where $K_{i,j}=0$ for $i\ge j$, $K_{i,j}=\psi_{i,j}(T_{i,i})T_{i,i+1}\cdots T_{j-1,j}$ for $j\ne n$ and $K_{i,n}=S_{i,n}$ for $1\le i\le n-1$, such that $X:=I+K$ is invertible and satisfies that $XT=SX$. This completes the proof.
\end{proof}

With the crucial bridge of Proposition \ref{lemma4.4}, we can now give a proof of the main theorem:

\textbf{Proof of the Theorem \ref{lemma4}:}
From the Proposition \ref{lemma4.4}, we know that there exist two invertible operators $X$ and $Y$ such that $XT=SX$ and $Y\tilde{T}=\tilde{S}Y$.

If $T\sim_{q.s}\tilde{T}$, then there exist two quasi-affine intertwiners $W$ and $Z$ such that $TW=W\tilde{T}$ and $ZT=\tilde{T}Z$. Together with $T\sim_{s}S$ and $\tilde{T}\sim_{s}\tilde{S}$, we get that
\begin{align*}
     XWY^{-1}\tilde{S}&=XW\tilde{T}Y^{-1}=XTWY^{-1}=SXWY^{-1},\\
     \tilde{S}YZX^{-1}&=Y\tilde{T}ZX^{-1}=YZTX^{-1}=YZX^{-1}S.
\end{align*}
Since both $X,Y$ are invertible and both $W,Z$ are quasi-affine, then both $XWY^{-1},YZX^{-1}$ are quasi-affine. Thus $S\sim_{q.s}\tilde{S}$.

Conversely, if  $S\sim_{q.s}\tilde{S}$, then there exist two quasi-affine intertwiners $W'$ and $Z'$ such that $SW'=W'\tilde{S}$ and $Z'S=\tilde{S}Z'$. Together with $T\sim_{s}S$ and $\tilde{T}\sim_{s}\tilde{S}$, we obtain that
\begin{align*}
	X^{-1}W'Y\tilde{T}&=X^{-1}W'\tilde{S}Y=X^{-1}SW'Y=TX^{-1}W'Y,\\
	\tilde{T}Y^{-1}Z'X&=Y^{-1}\tilde{S}Z'X=Y^{-1}Z'SX=Y^{-1}Z'XT.
\end{align*}
Since both $X,Y$ are invertible and both $W',Z'$ are quasi-affine, then both $X^{-1}W'Y$ and $Y^{-1}Z'X$ are quasi-affine. Thus $T\sim_{q.s}\tilde{T}$. This completes the proof.

\section{quasi-similarity for operators in $\mathcal{F}_{n}(\mathcal{A})$}\label{sec4}
In this section, the relationship between the similarity and quasi-similarity classification of the high-index operators, which were introduced in the previous section, will be considered.
When the class $\mathcal{A}$ is the set of ONB shifts (unless otherwise specified, the ONB shifts mentioned in this section are non quasi-nilpotent operators), we demonstrate that any two operators $T,\tilde{T}\in\mathcal{F}_{n}(\mathcal{A})$ with certain conditions are quasi-similar if and only if they are similar, which provides a partial answer to the D.A. Herrero's Question \ref{question1.2}.
Furthermore, as an application, the quasi-similarity and similarity of multiplication operators on some vector-valued reproducing kernel Hilbert spaces are characterized.
This characterization extends Proposition 2.6 of M. Uchiyama in \cite{Uchiyama} to the high-index cases.

The following is the main theorem of this section:

\begin{thm}\label{theorem4.1}
Let $T=\big((T_{i,j})\big)_{i,j=1}^{n},\tilde{T}=\big((\tilde{T}_{i,j})\big)_{i,j=1}^{n}\in\mathcal{F}_{n}(\mathcal{A})$ satisfy Condition (A), where
\begin{enumerate}
\item $T_{i,i}\in\mathcal{B}(\mathcal{H}_{i})$ (resp., $\tilde{T}_{i,i}\in\mathcal{B}(\tilde{\mathcal{H}}_{i})$) is backward weighted shift operator with nonzero weight sequence $\{w_{k}^{(i)}\}_{k\ge0}$ (resp., $\{\tilde{w}_{k}^{(i)}\}_{k\ge0}$) on the orthonormal basis $\{e^{(i)}_{k}\}_{k\ge0}$ of $\mathcal{H}_{i}$ (resp., $\{\tilde{e}^{(i)}_{k}\}_{k\ge0}$ of $\tilde{\mathcal{H}}_{i}$) for $1\leq i\leq n$;
\item For $1\le i\le n-1$, $T_{i,i+1}(e^{(i+1)}_{0})=e^{(i)}_{0}$ (resp., $\tilde{T}_{i,i+1}(\tilde{e}^{(i+1)}_{0})=\tilde{e}^{(i)}_{0}$) and for $k\ge1$, $$T_{i,i+1}(e^{(i+1)}_{k})=\frac{\prod\limits_{m=0}^{k-1}w^{(i+1)}_{m}}{\prod\limits_{m=0}^{k-1}w^{(i)}_{m}}e^{(i)}_{k},\ (\mbox{resp.,}\ \tilde{T}_{i,i+1}(\tilde{e}^{(i+1)}_{k})=\frac{\prod\limits_{m=0}^{k-1}\tilde{w}^{(i+1)}_{m}}{\prod\limits_{m=0}^{k-1}\tilde{w}^{(i)}_{m}}\tilde{e}^{(i)}_{k}).$$
\end{enumerate}
Then $T\sim_{q.s}\tilde{T}$ is equivalent to $T\sim_{s}\tilde{T}$.
\end{thm}

To prove Theorem \ref{theorem4.1}, from Theorem \ref{lemma4} and Proposition \ref{lemma4.4},
%we can know that, in the context of Condition (A), the similarity and the quasi-similarity of the operators in $\mathcal{F}_{n}(\mathcal{A})$ can be fully converted to the similarity and the quasi-similarity of the operators in $\mathcal{OF}_{n}(\mathcal{A})$, respectively. Thus
it suffices to consider the case where $T$ and $\tilde{T}$ given in Theorem \ref{theorem4.1} belong to the class $\mathcal{OF}_{n}(\mathcal{A})$. Next, we will use mathematical induction to prove it.

%Although the operators in $\mathcal{OF}_{n}(\mathcal{A})$ have a clear and simple structure, when considering $\mathcal{A}$ as the class of Cowen-Douglas operators with index 1, they are norm-dense up to similarity.

\begin{lem}\label{lemma4.6}
Let $T=\left(\begin{smallmatrix}T_{1,1}&T_{1,2}\\0&T_{2,2}\end{smallmatrix}\right), \tilde{T}=\left(\begin{smallmatrix}\tilde{T}_{1,1}&\tilde{T}_{1,2}\\0&\tilde{T}_{2,2}\end{smallmatrix}\right)\in\mathcal{F}_{2}(\mathcal{A})$, where $T_{i,i}\in\mathcal{B}(\mathcal{H}_{i})$ (resp., $\tilde{T}_{i,i}\in\mathcal{B}(\mathcal{\tilde{H}}_{i})$) is backward weighted shift operator with nonzero weight sequence $\{w_{k}^{(i)}\}_{k\ge0}$ (resp., $\{\tilde{w}_{k}^{(i)}\}_{k\ge0}$) on the orthonormal basis $\{e^{(i)}_{k}\}_{k\ge0}$ of $\mathcal{H}_{i}$ (resp., $\{\tilde{e}^{(i)}_{k}\}_{k\ge0}$ of $\tilde{\mathcal{H}}_{i}$) for $i=1,2$.
Suppose that
\begin{equation*}
\setlength\abovedisplayskip{3pt}
\setlength\belowdisplayskip{3pt}
T_{1,2}(e^{(2)}_{0})=e^{(1)}_{0},\ \tilde{T}_{1,2}(\tilde{e}^{(2)}_{0})=\tilde{e}^{(1)}_{0}\ \mbox{and}\ T_{1,2}(e^{(2)}_{k})=\frac{\prod\limits_{m=0}^{k-1}w^{(2)}_{m}}{\prod\limits_{m=0}^{k-1}w^{(1)}_{m}}e^{(1)}_{k},\ \tilde{T}_{1,2}(\tilde{e}^{(2)}_{k})=\frac{\prod\limits_{m=0}^{k-1}\tilde{w}^{(2)}_{m}}{\prod\limits_{m=0}^{k-1}\tilde{w}^{(1)}_{m}}\tilde{e}^{(1)}_{k}
\end{equation*}
for $k\ge1$. Then $T\sim_{q.s}\tilde{T}$ is equivalent to $T\sim_{s}\tilde{T}$.
\end{lem}
\begin{proof}
It is obvious that $T\sim_{s}\tilde{T}$ implies that $T\sim_{q.s}\tilde{T}$. So next we only need to prove the other direction. Assume that $T\sim_{q.s}\tilde{T}$, i.e. there exist two quasi-affine intertwining operators $X$ and $Y$ such that $TX=X\tilde{T}$ and $YT=\tilde{T}Y$.
Then, by Proposition \ref{proposition3.7}, $X$ and $Y$ are in the form of upper-triangular
$X=\left(\begin{smallmatrix}X_{1,1}&X_{1,2}\\0&X_{2,2}\end{smallmatrix}\right)$ and $Y=\left(\begin{smallmatrix}Y_{1,1}&Y_{1,2}\\0&Y_{2,2}\end{smallmatrix}\right)$ such that
\begin{equation*}
\setlength\abovedisplayskip{3pt}
\setlength\belowdisplayskip{3pt}
\left(\begin{smallmatrix}T_{1,1}&T_{1,2}\\0&T_{2,2}\end{smallmatrix}\right)
\left(\begin{smallmatrix}X_{1,1}&X_{1,2}\\0&X_{2,2}\end{smallmatrix}\right)
=\left(\begin{smallmatrix}X_{1,1}&X_{1,2}\\0&X_{2,2}\end{smallmatrix}\right)
\left(\begin{smallmatrix}\tilde{T}_{1,1}&\tilde{T}_{1,2}\\0&\tilde{T}_{2,2}\end{smallmatrix}\right)
\end{equation*}
and
\begin{equation*}
\setlength\abovedisplayskip{3pt}
\setlength\belowdisplayskip{3pt}
\left(\begin{smallmatrix}Y_{1,1}&Y_{1,2}\\0&Y_{2,2}\end{smallmatrix}\right)
\left(\begin{smallmatrix}T_{1,1}&T_{1,2}\\0&T_{2,2}\end{smallmatrix}\right)
=\left(\begin{smallmatrix}\tilde{T}_{1,1}&\tilde{T}_{1,2}\\0&\tilde{T}_{2,2}\end{smallmatrix}\right)
\left(\begin{smallmatrix}Y_{1,1}&Y_{1,2}\\0&Y_{2,2}\end{smallmatrix}\right).
\end{equation*}
By calculating the left and right sides of the above equation using matrix multiplication, we can obtain that $T_{1,1}X_{1,1}=X_{1,1}\tilde{T}_{1,1}$ and $Y_{1,1}T_{1,1}=\tilde{T}_{1,1}Y_{1,1}$.
For one thing, the injectivity of $X$ and $Y$ implies that $X_{1,1}$ and $Y_{1,1}$ are also injective.
For another, it follows from Proposition \ref{proposition3.2} that the ranges of $X_{1,1}$ and $Y_{1,1}$ are both dense. Consequently, $X_{1,1}$ and $Y_{1,1}$ are quasi-affine and satisfy $T_{1,1}\sim_{q.s}\tilde{T}_{1,1}$, which, according to Lemma \ref{lemma2.5}, is equivalent to $T_{1,1}\sim_{s}\tilde{T}_{1,1}$.

From Lemma \ref{lemma4.5}, we know that $X_{k,k}=\big((x^{(k)}_{i,j})\big)_{i,j\geq0}$ and $Y_{k,k}=\big((y^{k}_{i,j})\big)_{i,j\geq0}$ are in the form of upper-triangular with respect to $\{e^{(k)}_{l}\}_{l\ge0}$, $\{\tilde{e}^{(k)}_{l}\}_{l\ge0}$, respectively, where
\begin{equation*}
\setlength\abovedisplayskip{3pt}
\setlength\belowdisplayskip{3pt}
x^{(k)}_{l+1,l+1}=\frac{\prod\limits_{i=0}^{l}\tilde{w}^{(k)}_{i}}{\prod\limits_{i=0}^{l}w^{(k)}_{i}}x^{(k)}_{0,0},\  x^{(k)}_{l+1,l+j}=\frac{\prod\limits_{i=j-1}^{l+j-1}\tilde{w}^{(k)}_{i}}{\prod\limits_{i=0}^{l}w^{(k)}_{i}}x^{(k)}_{0,j-1},\
j\ge2,\ l\ge0,\ k=1,2
\end{equation*}
and
\begin{equation*}
\setlength\abovedisplayskip{3pt}
\setlength\belowdisplayskip{3pt}
y^{(k)}_{l+1,l+1}=\frac{\prod\limits_{i=0}^{l}w^{(k)}_{i}}{\prod\limits_{i=0}^{l}\tilde{w}^{(k)}_{i}}y^{(k)}_{0,0},\
y^{(k)}_{l+1,l+j}=\frac{\prod\limits_{i=j-1}^{l+j-1}w^{(k)}_{i}}{\prod\limits_{i=0}^{l}\tilde{w}^{(k)}_{i}}y^{(k)}_{0,j-1},\
j\ge2,\ l\ge0,\ k=1,2.
\end{equation*}
Similarly, we can assume that $X_{1,2}(\tilde{e}^{(2)}_{l})=\sum\limits_{k=0}^{\infty}x^{(1,2)}_{k,l}e^{(1)}_{k}$ and $Y_{1,2}(e^{(2)}_{l})=\sum\limits_{k=0}^{\infty}y^{(1,2)}_{k,l}\tilde{e}^{(1)}_{k}$ for $l\ge0$. It's worthy to note that the condition that both $X_{1,1}$ and $Y_{1,1}$ are injective implies that $x^{(1)}_{0,0}\ne0$ and $y^{(1)}_{0,0}\ne0$.

Next we prove by contraction that $x^{(2)}_{0,0}\ne0$ and $y^{(2)}_{0,0}\ne0$. Suppose that $x^{(2)}_{0,0}=0$ and $y^{(2)}_{0,0}=0$, which means that $X_{2,2}(\tilde{e}^{(2)}_{0})=0$ and $Y_{2,2}(e^{(2)}_{0})=0$.
We then act the left and right sides of the following two equations on the vectors $\tilde{e}^{(2)}_{0}$ and $e^{(2)}_{0}$, respectively,
\begin{equation*}
\setlength\abovedisplayskip{3pt}
\setlength\belowdisplayskip{3pt}
T_{1,1}X_{1,2}-X_{1,2}\tilde{T}_{2,2}=X_{1,1}\tilde{T}_{1,2}-T_{1,2}X_{2,2}\ \mbox{and}\ Y_{1,2}T_{2,2}-\tilde{T}_{1,1}Y_{1,2}=\tilde{T}_{1,2}Y_{2,2}-Y_{1,1}T_{1,2}.
\end{equation*}
Through routine calculations and by comparing the coefficients, we find that
\begin{equation*}
\setlength\abovedisplayskip{3pt}
\setlength\belowdisplayskip{3pt}
\sum\limits_{i=0}^{\infty}x^{(1,2)}_{i+1,0}w^{(1)}_{i}e^{(1)}_{i}=x^{(1)}_{0,0}e^{(1)}_{0}\ \mbox{and}\ \sum\limits_{i=0}^{\infty}y^{(1,2)}_{i+1,0}\tilde{w}^{(1)}_{i}\tilde{e}^{(1)}_{i}=y^{(1)}_{0,0}\tilde{e}^{(1)}_{0}.
\end{equation*}
Thus, we can see that $x^{(1,2)}_{1,0}=\frac{x^{(1)}_{0,0}}{w^{(1)}_{0}}\ne0$, $x^{(1,2)}_{i,0}=0$, $i\ge2$, and
$y^{(1,2)}_{1,0}=\frac{y^{(1)}_{0,0}}{\tilde{w}^{(1)}_{0}}\ne0$, $y^{(1,2)}_{i,0}=0$, $i\ge2$.
The above relations show that $X_{1,2}(\tilde{e}^{(2)}_{0})=x^{(1,2)}_{0,0}e^{(1)}_{0}+x^{(1,2)}_{1,0}e^{(1)}_{1}\ne0$ and $Y_{1,2}(e^{(2)}_{0})=y^{(1,2)}_{0,0}\tilde{e}^{(1)}_{0}+y^{(1,2)}_{1,0}\tilde{e}^{(1)}_{1}\ne0$. Furthermore, we obtain
\begin{align*}
X_{1,1}\Big(\frac{x^{(1,2)}_{0,0}}{x^{(1)}_{0,0}}\tilde{e}^{(1)}_{0}-\frac{x^{(1)}_{0,1}}{\tilde{w}^{(1)}_{0}x^{(1)}_{0,0}}\tilde{e}^{(1)}_{0}+\frac{1}{\tilde{w}^{(1)}_{0}}\tilde{e}^{(1)}_{1}\Big)
&=x^{(1,2)}_{0,0}e^{(1)}_{0}-\frac{x^{(1)}_{0,1}}{\tilde{w}^{(1)}_{0}}e^{(1)}_{0}+\frac{1}{\tilde{w}^{(1)}_{0}}(x^{(1)}_{0,1}e^{(1)}_{0}+x^{(1)}_{1,1}e^{(1)}_{1})\\
&=x^{(1,2)}_{0,0}e^{(1)}_{0}+\frac{x^{(1)}_{1,1}}{\tilde{w}^{(1)}_{0}}e^{(1)}_{1}\\
&=x^{(1,2)}_{0,0}e^{(1)}_{0}+\frac{x^{(1)}_{0,0}}{w^{(1)}_{0}}e^{(1)}_{1}\\
&=X_{1,2}(\tilde{e}^{(2)}_{0}),
\end{align*}
and
\begin{align*}
Y_{1,1}\Big(\frac{y^{(1,2)}_{0,0}}{y^{(1)}_{0,0}}e^{(1)}_{0}-\frac{y^{(1)}_{0,1}}{w^{(1)}_{0}y^{(1)}_{0,0}}e^{(1)}_{0}+\frac{1}{w^{(1)}_{0}}e^{(1)}_{1}\Big)
&=y^{(1,2)}_{0,0}\tilde{e}^{(1)}_{0}-\frac{y^{(1)}_{0,1}}{w^{(1)}_{0}}\tilde{e}^{(1)}_{0}+\frac{1}{w^{(1)}_{0}}(y^{(1)}_{0,1}\tilde{e}^{(1)}_{0}+y^{(1)}_{1,1}\tilde{e}^{(1)}_{1})\\
&=y^{(1,2)}_{0,0}\tilde{e}^{(1)}_{0}+\frac{y^{(1)}_{1,1}}{w^{(1)}_{0}}\tilde{e}^{(1)}_{1}\\
&=y^{(1,2)}_{0,0}\tilde{e}^{(1)}_{0}+\frac{y^{(1)}_{0,0}}{\tilde{w}^{(1)}_{0}}\tilde{e}^{(1)}_{1}\\
&=Y_{1,2}(e^{(2)}_{0}),
\end{align*}
which imply that $\mbox{ran}\,X_{1,1}\cap\mbox{ran}\,X_{1,2}|_{\ker X_{2,2}}\ne\{0\}\ \mbox{and}\ \mbox{ran}\,Y_{1,1}\cap\mbox{ran}\,Y_{1,2}|_{\ker Y_{2,2}}\ne\{0\}$.
This contradicts the fact that both $X$ and $Y$ are injective. Hence, $x^{(2)}_{0,0}\ne0$ and $y^{(2)}_{0,0}\ne0$.
Subsequently, we have that $0<\frac{\vert y^{(2)}_{0,0}\vert}{\Vert Y_{2,2}\Vert}\le\Big|\frac{\prod\limits_{i=0}^{l}\tilde{w}^{(2)}_{i}}{\prod\limits_{i=0}^{l}w^{(2)}_{i}}\Big|\le\frac{\Vert X_{2,2}\Vert}{\vert x^{(2)}_{0,0}\vert}<\infty$ for all $l\ge0$. By A.L. Shields's result in \cite{Shields}, it means that $T_{2,2}\sim_{s}\tilde{T}_{2,2}$.

Define the operator $\tilde{X}_{i,i}:\mathcal{\tilde{H}}_{i}\rightarrow\mathcal{H}_{i}$, $i=1,2$, by $\tilde{X}_{i,i}(\tilde{e}^{(i)}_{0})=e^{(i)}_{0}$ and for $l\ge1$, $\tilde{X}_{i,i}(\tilde{e}^{(i)}_{l})=\frac{\prod\limits_{m=0}^{l-1}\tilde{w}^{(i)}_{m}}{\prod\limits_{m=0}^{l-1}w^{(i)}_{m}}e^{(i)}_{l}$.
For $i=1,2$, the relationship $T_{i,i}\sim_{s}\tilde{T}_{i,i}$ implies that the sequence $\{1,\frac{\tilde{w}^{(i)}_{0}}{w^{(i)}_{0}},\cdots,\frac{\prod\limits_{m=0}^{l-1}\tilde{w}^{(i)}_{m}}{\prod\limits_{m=0}^{l-1}w^{(i)}_{m}},\cdots\}$
is bounded and bounded low, which concludes that the operator $\tilde{X}_{i,i}$ is invertible. It's obvious that
\begin{equation*}
\setlength\abovedisplayskip{3pt}
\setlength\belowdisplayskip{3pt}
T_{1,2}\tilde{X}_{2,2}(\tilde{e}^{(2)}_{0})=e^{(1)}_{0}=\tilde{X}_{1,1}\tilde{T}_{1,2}(\tilde{e}^{(2)}_{0})\ \mbox{and}\
T_{k,k}\tilde{X}_{k,k}(\tilde{e}^{(k)}_{0})=0=\tilde{X}_{k,k}\tilde{T}_{k,k}(\tilde{e}^{(k)}_{0}),\ k=1,2.
\end{equation*}
Moreover, for all $l\ge1$, we have
\begin{small}
\begin{equation*}
\setlength\abovedisplayskip{3pt}
\setlength\belowdisplayskip{3pt}
T_{1,2}\tilde{X}_{2,2}(\tilde{e}^{(2)}_{l})
=T_{1,2}\Big(\frac{\prod\limits_{i=0}^{l-1}\tilde{w}^{(2)}_{i}}{\prod\limits_{i=0}^{l-1}w^{(2)}_{i}}e^{(2)}_{l}\Big)
=\frac{\prod\limits_{i=0}^{l-1}w^{(2)}_{i}}{\prod\limits_{i=0}^{l-1}w^{(1)}_{i}}\frac{\prod\limits_{i=0}^{l-1}\tilde{w}^{(2)}_{i}}{\prod\limits_{i=0}^{l-1}w^{(2)}_{i}}e^{(1)}_{l}
=\tilde{X}_{1,1}\Big(\frac{\prod\limits_{i=0}^{l-1}\tilde{w}^{(2)}_{i}}{\prod\limits_{i=0}^{l-1}\tilde{w}^{(1)}_{i}}\tilde{e}^{(1)}_{l}\Big)
=\tilde{X}_{1,1}\tilde{T}_{1,2}(\tilde{e}^{(2)}_{l})
\end{equation*}
\end{small}
and for $k=1,2$,
\begin{small}
\begin{equation*}
\setlength\abovedisplayskip{3pt}
\setlength\belowdisplayskip{3pt}
T_{k,k}\tilde{X}_{k,k}(\tilde{e}^{(k)}_{l})
=T_{k,k}\Big(\frac{\prod\limits_{i=0}^{l-1}\tilde{w}^{(k)}_{i}}{\prod\limits_{i=0}^{l-1}w^{(k)}_{i}}e^{(k)}_{l}\Big)
=\frac{\prod\limits_{i=0}^{l-1}\tilde{w}^{(k)}_{i}}{\prod\limits_{i=0}^{l-1}w^{(k)}_{i}}w^{(k)}_{l-1}e^{(k)}_{l-1}
=\tilde{X}_{k,k}(\tilde{w}^{(k)}_{l-1}e^{(1)}_{l-1})
=\tilde{X}_{k,k}\tilde{T}_{k,k}(\tilde{e}^{(k)}_{l}).
\end{equation*}
\end{small}
Therefore we get that $T_{k,k}\tilde{X}_{k,k}=\tilde{X}_{k,k}\tilde{T}_{k,k}$, $k=1,2$, and $T_{1,2}\tilde{X}_{2,2}=\tilde{X}_{1,1}\tilde{T}_{1,2}$. That is to say that the invertible operator $\tilde{X}:=\left(\begin{smallmatrix}\tilde{X}_{1,1}&0\\0&\tilde{X}_{2,2}\end{smallmatrix}\right)$ satisfies $T\tilde{X}=\tilde{X}\tilde{T}$, i.e., $T\sim_{s}\tilde{T}$.
\end{proof}

\begin{prop}\label{proposition4.8}
Let $T=\big((T_{i,j})\big)_{i,j=1}^{n}, \tilde{T}=\big((\tilde{T}_{i,j})\big)_{i,j=1}^{n}\in\mathcal{OF}_{n}(\mathcal{A})$, where
\begin{enumerate}
\item $T_{i,i}\in\mathcal{B}(\mathcal{H}_{i})$ (resp., $\tilde{T}_{i,i}\in\mathcal{B}(\tilde{\mathcal{H}}_{i})$) is backward weighted shift operator with nonzero weight sequence $\{w_{k}^{(i)}\}_{k\ge0}$ (resp., $\{\tilde{w}_{k}^{(i)}\}_{k\ge0}$) on the orthonormal basis $\{e^{(i)}_{k}\}_{k\ge0}$ of $\mathcal{H}_{i}$ (resp., $\{\tilde{e}^{(i)}_{k}\}_{k\ge0}$ of $\tilde{\mathcal{H}}_{i}$) for $1\leq i\leq n$;
\item For $1\le i\le n-1$, $T_{i,i+1}(e^{(i+1)}_{0})=e^{(i)}_{0}$ (resp., $\tilde{T}_{i,i+1}(\tilde{e}^{(i+1)}_{0})=\tilde{e}^{(i)}_{0}$) and for $k\ge1$, $$T_{i,i+1}(e^{(i+1)}_{k})=\frac{\prod\limits_{m=0}^{k-1}w^{(i+1)}_{m}}{\prod\limits_{m=0}^{k-1}w^{(i)}_{m}}e^{(i)}_{k},\ (\mbox{resp.,}\ \tilde{T}_{i,i+1}(\tilde{e}^{(i+1)}_{k})=\frac{\prod\limits_{m=0}^{k-1}\tilde{w}^{(i+1)}_{m}}{\prod\limits_{m=0}^{k-1}\tilde{w}^{(i)}_{m}}\tilde{e}^{(i)}_{k}).$$
\end{enumerate}
Then $T\sim_{q.s}\tilde{T}$ is equivalent to $T\sim_{s}\tilde{T}$.
\end{prop}
\begin{proof}
It's clear that $T\sim_{s}\tilde{T}$ implies that $T\sim_{q.s}\tilde{T}$. So next we only need to show that if there exist two quasi-affine intertwining operators $X=\big((X_{i,j})\big)_{i,j=1}^{n}$ and $Y=\big((Y_{i,j})\big)_{i,j=1}^{n}$, such that $TX=X\tilde{T}$ and $YT=\tilde{T}Y$, then $T\sim_{s}\tilde{T}$.
By Proposition \ref{proposition3.7}, we know that $X$ and $Y$ are upper-triangular.

From Lemma \ref{lemma4.5}, we have that $X_{k,k}=\big((x^{(k)}_{i,j})\big)_{i,j\geq0}$, $1\le k\le n$, where $x^{(k)}_{i,j}=0$ for $1\le j<i\le n$ and
\begin{equation*}
\setlength\abovedisplayskip{3pt}
\setlength\belowdisplayskip{3pt}
x^{(k)}_{l+1,l+1}=\frac{\prod\limits_{i=0}^{l}\tilde{w}^{(k)}_{i}}{\prod\limits_{i=0}^{l}w^{(k)}_{i}}x_{0,0}^{(k)},\
x^{(k)}_{l+1,l+j}=\frac{\prod\limits_{i=j-1}^{l+j-1}\tilde{w}^{(k)}_{i}}{\prod\limits_{i=0}^{l}w^{(k)}_{i}}x^{(k)}_{0,j-1},\ j\ge2,\ l\ge0.
\end{equation*}
Similarly, we can assume that $X_{i,j}(\tilde{e}^{(j)}_{l})=\sum\limits_{k=0}^{\infty}x^{(i,j)}_{k,l}e^{(i)}_{k}$ and $Y_{i,j}(e^{(j)}_{l})=\sum\limits_{k=0}^{\infty}y^{(i,j)}_{k,l}\tilde{e}^{(i)}_{k}$ for $l\ge0$ and $1\le i<i+1\le j\le n$. Since $X$ is injective, so is $X_{1,1}$, and thus $x^{(1)}_{0,0}\ne0$. Furthermore, we will use mathematical induction to prove the following claim for $n$:
\begin{claim}
$x^{(k)}_{0,0}\ne0$ for all $1\le k\le n$.
\end{claim}
The validity of the case $n=2$, is immediate from the proof of Lemma \ref{lemma4.6}. Suppose that $x^{(k)}_{0,0}\ne0$ for all $1\le k\le m$ holds for all $m<n$. Then for the case of $n$, the induction hypothesis guarantees that $x^{(k)}_{0,0}\ne0$ for $1\le k\le n-1$. Next we will show that $x^{(n)}_{0,0}\ne0$ by contradiction.
If $x^{(n)}_{0,0}=0$, the following fact would hold:
\begin{fact}
$X_{i,j}(\tilde{e}^{(j)}_{l})=\sum\limits_{k=0}^{j-i+l}x^{(i,j)}_{k,l}e^{(i)}_{k}$ for all $l\ge0$ and $1\le i<j\le n$.
\end{fact}
Specific certifications are as follows. Firstly, we calculate the coefficients of the following relations acting on the vector $\tilde{e}^{(i+1)}_{0}$:
\begin{equation}\label{equation4.1}
T_{i,i}X_{i,i+1}-X_{i,i+1}\tilde{T}_{i+1,i+1}=X_{i,i}\tilde{T}_{i,i+1}-T_{i,i+1}X_{i+1,i+1},\ 1\le i\le n-1.
\end{equation}
We have that
\begin{equation*}
\setlength\abovedisplayskip{3pt}
\setlength\belowdisplayskip{3pt}
\sum\limits_{k=0}^{\infty}x^{(i,i+1)}_{k+1,0}w^{(i)}_{k}e^{(i)}_{k}=\left\{
\begin{aligned}
&(x^{(i)}_{0,0}-x^{(i+1)}_{0,0})e^{(i)}_{0},\ 1\le i\le n-2;\\
&x^{(n-1)}_{0,0}e^{(n-1)}_{0},\ i=n-1.
\end{aligned}
\right.
\end{equation*}
%If $x^{(i)}_{0,0}=x^{(i+1)}_{0,0}$ for $1\le i\le n-2$, then $X_{i,i+1}(\tilde{e}^{(i+1)}_{0})=0$.
Therefore, $x^{(i,i+1)}_{k,0}=0$ for all $k\ge2$ and $1\le i\le n-1$ and $x^{(n-1,n)}_{1,0}\ne0$. Then we consider the coefficients of the relations (\ref{equation4.1}) acting on the vector $\tilde{e}^{(i+1)}_{1}$:
\begin{equation*}
\setlength\abovedisplayskip{3pt}
\setlength\belowdisplayskip{3pt}
\begin{aligned}
\sum\limits_{k=0}^{\infty}\big(x^{(i,i+1)}_{k+1,1}w^{(i)}_{k}-x^{(i,i+1)}_{k,0}\tilde{w}^{(i+1)}_{0}\big)e^{(i)}_{k}
&=\big(x^{(i)}_{0,1}\frac{\tilde{w}^{(i+1)}_{0}}{\tilde{w}^{(i)}_{0}}-x^{(i+1)}_{0,1}\big)e^{(i)}_{0}\\
&+\big(x^{(i)}_{1,1}\frac{\tilde{w}^{(i+1)}_{0}}{\tilde{w}^{(i)}_{0}}-x^{(i+1)}_{1,1}\frac{w^{(i+1)}_{0}}{w^{(i)}_{0}}\big)e^{(i)}_{1}.
\end{aligned}
\end{equation*}
Then $x^{(i,i+1)}_{k+1,1}w^{(i)}_{k}=x^{(i,i+1)}_{k,0}\tilde{w}^{(i+1)}_{0}$ for $k\ge2$. Since $x^{(i,i+1)}_{k,0}=0$ for $k\ge2$, $x^{(i,i+1)}_{k,1}=0$ for $k\ge3$.
Suppose that $x^{(i,i+1)}_{k,l}=0$ for $k\ge l+2$, then we prove that $x^{(i,i+1)}_{k,l+1}=0$ for $k\ge l+3$.
Considering the coefficients of the relation (\ref{equation4.1}) acting on the vector $\tilde{e}^{(i+1)}_{l+1}$, we obtain
\begin{equation*}
\setlength\abovedisplayskip{3pt}
\setlength\belowdisplayskip{3pt}
\begin{aligned}
\sum\limits_{k=0}^{\infty}\big(x^{(i,i+1)}_{k+1,l+1}w^{(i)}_{k}-x^{(i,i+1)}_{k,l}\tilde{w}^{(i+1)}_{l}\big)e^{(i)}_{k}
&=\big(x^{(i)}_{0,l+1}\frac{\prod\limits_{m=0}^{l}\tilde{w}^{(i+1)}_{m}}{\prod\limits_{m=0}^{l}\tilde{w}^{(i)}_{m}}-x^{(i+1)}_{0,l+1}\big)e^{(i)}_{0}\\
&+\sum\limits_{k=1}^{l+1}\big(x^{(i)}_{k,l+1}\frac{\prod\limits_{m=0}^{l}\tilde{w}^{(i+1)}_{m}}{\prod\limits_{m=0}^{l}\tilde{w}^{(i)}_{m}}-x^{(i+1)}_{k,l+1}\frac{\prod\limits_{m=0}^{k-1}w^{(i+1)}_{m}}{\prod\limits_{m=0}^{k-1}w^{(i)}_{m}}\big)e^{(i)}_{k},
\end{aligned}
\end{equation*}
which shows that $x^{(i,i+1)}_{k+1,l+1}w^{(i)}_{k}=x^{(i,i+1)}_{k,l}\tilde{w}^{(i+1)}_{l}$ for $k\ge l+2$. Then by the hypothesis, we obtain that $x^{(i,i+1)}_{k,l+1}=0$ for $k\ge l+3$.
As we have illustrated above, for $1\le i\le n-1$, we have that $X_{i,i+1}(\tilde{e}^{(i+1)}_{l})=\sum\limits_{k=0}^{l+1}x^{(i,i+1)}_{k,l}e^{(i)}_{k}$ for $l\ge0$.

We now assume that $X_{i,i+j}(\tilde{e}^{(i+j)}_{l})=\sum\limits_{k=0}^{j+l}x^{(i,i+j)}_{k,l}e^{(i)}_{k}$ for all $l\ge0$ and $1\le i< i+j\le n-1$, in order to prove that  $X_{i,i+j+1}(\tilde{e}^{(i+j+1)}_{l})=\sum\limits_{k=0}^{j+1+l}x^{(i,i+j+1)}_{k,l}e^{(i)}_{k}$.
We will consider the following relationship acting on the vector $\tilde{e}^{(i+j+1)}_{0}$:
\begin{equation}\label{equation4.3}
\setlength\abovedisplayskip{3pt}
\setlength\belowdisplayskip{3pt}
T_{i,i}X_{i,i+j+1}-X_{i,i+j+1}\tilde{T}_{i+j+1,i+j+1}=X_{i,i+j}\tilde{T}_{i+j,i+j+1}-T_{i,i+1}X_{i+1,i+j+1}.
\end{equation}
It follows that
\begin{equation*}
\setlength\abovedisplayskip{3pt}
\setlength\belowdisplayskip{3pt}
\begin{aligned}
\sum\limits_{k=0}^{\infty}x^{(i,i+j+1)}_{k+1,0}w^{(i)}_{k}e^{(i)}_{k}
&=\big(x^{(i,i+j)}_{0,0}-x^{(i+1,i+j+1)}_{0,0}\big)e^{(i)}_{0}\\
&+\sum\limits_{k=1}^{j}\big(x^{(i,i+j)}_{k,0}-x^{(i+1,i+j+1)}_{k,0}\frac{\prod\limits_{m=0}^{k-1}w^{(i+1)}_{m}}{\prod\limits_{m=0}^{k-1}w^{(i)}_{m}}\big)e^{(i)}_{k},
\end{aligned}
\end{equation*}
which end up with $x^{(i,i+j+1)}_{k,0}=0$ for $k\ge j+2$ and $1\le i<i+j\le n-1$.
Suppose that $x^{(i,i+j+1)}_{k,l}=0$ for $k\ge j+l+1$, and we will show that $x^{(i,i+j+1)}_{k,l+1}=0$ for $k\ge j+l+2$.
The image of equation (\ref{equation4.3}) acting on the vector $\tilde{e}^{(i+j+1)}_{l+1}$ is
\begin{equation*}
\setlength\abovedisplayskip{3pt}
\setlength\belowdisplayskip{3pt}
\begin{aligned}
\sum\limits_{k=0}^{\infty}\big(x^{(i,i+j+1)}_{k+1,l+1}w^{(i)}_{k}-x^{(i,i+j+1)}_{k,l}\tilde{w}^{(i+j+1)}_{l}\big)e^{(i)}_{k}
&=\big(x^{(i,i+j)}_{0,l}-x^{(i+1,i+j+1)}_{0,l}\big)e^{(i)}_{0}\\
&+\sum\limits_{k=1}^{j+l}\big(x^{(i,i+j)}_{k,l}-x^{(i+1,i+j+1)}_{k,l}\frac{\prod\limits_{m=0}^{k-1}w^{(i+1)}_{m}}{\prod\limits_{m=0}^{k-1}w^{(i)}_{m}}\big)e^{(i)}_{k},
\end{aligned}
\end{equation*}
which concludes that $x^{(i,i+j+1)}_{k+1,l+1}w^{(i)}_{k}=x^{(i,i+j+1)}_{k,l}\tilde{w}^{(i+j+1)}_{l}$ for $k\ge j+1+l$. Thus, $x^{(i,i+j+1)}_{k,l+1}=0$ for $k\ge j+l+2$. As we have shown above, for $1\le i< i+j\le n-1$, we obtain that $X_{i,i+j+1}(\tilde{e}^{(i+j+1)}_{l})=\sum\limits_{k=0}^{j+1+l}x^{(i,i+j+1)}_{k,l}e^{(i)}_{k}$ for $l\ge0$. This completes the proof of the Fact.
\begin{comment}
$$
X_{i,i+j+1}=\left(\begin{smallmatrix}
		x^{(i,i+j+1)}_{0,0}&x^{(i,i+j+1)}_{0,1}&\cdots&\cdots&\cdots&x^{(i,i+j+1)}_{0,l-1}&x^{(i,i+j+1)}_{0,l}&\cdots\\
		x^{(i,i+j+1)}_{1,0}&x^{(i,i+j+1)}_{1,1}&\cdots&\cdots&\cdots&x^{(i,i+j+1)}_{1,l-1}&x^{(i,i+j+1)}_{1,l}&\cdots\\
		\vdots&\vdots&\cdots&\cdots&\cdots&\vdots&\vdots&\cdots\\
        x^{(i,i+j+1)}_{j+1,0}&x^{(i,i+j+1)}_{j+1,1}&\cdots&\cdots&\cdots&x^{(i,i+j+1)}_{j+1,l-1}&x^{(i,i+j+1)}_{j+1,l}&\cdots\\
        &x^{(i,i+j+1)}_{j+2,1}&\cdots&\cdots&\cdots&x^{(i,i+j+1)}_{j+2,l-1}&x^{(i,i+j+1)}_{j+2,l}&\cdots\\
        &&\ddots&\ddots&\ddots&\vdots&\vdots&\cdots\\
		&&x^{(i,i+j+1)}_{l-1,l-j-2}&\cdots&\cdots&x^{(i,i+j+1)}_{l-1,l-1}&x^{(i,i+j+1)}_{l-1,l}&\cdots\\
		&&&x^{(i,i+j+1)}_{l,l-j-1}&\cdots&\cdots&x^{(i,i+j+1)}_{l,l}&\cdots\\
		&&&&&\ddots&\ddots
	\end{smallmatrix}\right).
$$
\end{comment}

Now, we continue to prove the claim. Our purpose is to find elements $v_{i}\in\tilde{\mathcal{H}}_{i}$ for $1\le i\le n-1$ such that $v=(v_{1},\cdots,v_{n-1},\tilde{e}^{(n)}_{0})\textsuperscript{T}$ is in the kernel of the operator $X$, that is,
\begin{equation*}
\setlength\abovedisplayskip{3pt}
\setlength\belowdisplayskip{3pt}
\begin{pmatrix}
X_{1,1}&X_{1,2}&X_{1,3}&\cdots&X_{1,n}\\
0&X_{2,2}&X_{2,3}&\cdots&X_{2,n}\\
\vdots&\ddots&\ddots&\ddots&\vdots\\
0&\cdots&0&X_{n-1,n-1}&X_{n-1,n}\\
0&\cdots&\cdots&0&X_{n,n}
\end{pmatrix}
\begin{pmatrix}
v_{1}\\
v_{2}\\
\vdots\\
v_{n-1}\\
\tilde{e}^{(n)}_{0}
\end{pmatrix}
=\begin{pmatrix}
0\\
0\\
\vdots\\
0\\
0
\end{pmatrix}.
\end{equation*}
Since $X_{n-1,n}(\tilde{e}^{(n)}_{0})=x^{(n-1,n)}_{0,0}e^{(n-1)}_{0}+x^{(n-1,n)}_{1,0}e^{(n-1)}_{1}\ne0$, it can be seen that
\begin{align*}
&X_{n-1,n-1}\Big(\frac{x^{(n-1,n)}_{0,0}}{x^{(n-1)}_{0,0}}\tilde{e}^{(n-1)}_{0}-\frac{x^{(n-1)}_{0,1}}{\tilde{w}^{(n-1)}_{0}x^{(n-1)}_{0,0}}\tilde{e}^{(n-1)}_{0}+\frac{1}{\tilde{w}^{(n-1)}_{0}}\tilde{e}^{(n-1)}_{1}\Big)\\
&=x^{(n-1,n)}_{0,0}e^{(n-1)}_{0}-\frac{x^{(n-1)}_{0,1}}{\tilde{w}^{(n-1)}_{0}}e^{(n-1)}_{0}+\frac{1}{\tilde{w}^{(n-1)}_{0}}\big(x^{(n-1)}_{0,1}e^{(n-1)}_{0}+x^{(n-1)}_{1,1}e^{(n-1)}_{1}\big)\\
&=x^{(n-1,n)}_{0,0}e^{(n-1)}_{0}+\frac{x^{(n-1)}_{1,1}}{\tilde{w}^{(n-1)}_{0}}e^{(n-1)}_{1}\\
&=x^{(n-1,n)}_{0,0}e^{(n-1)}_{0}+\frac{x^{(n-1)}_{0,0}}{w^{(n-1)}_{0}}e^{(n-1)}_{1}\\
&=x^{(n-1,n)}_{0,0}e^{(n-1)}_{0}+x^{(n-1,n)}_{1,0}e^{(n-1)}_{1}\\
&=X_{n-1,n}(\tilde{e}^{(n)}_{0}).
\end{align*}
Let $v_{n-1}:=-\frac{x^{(n-1,n)}_{0,0}}{x^{(n-1)}_{0,0}}\tilde{e}^{(n-1)}_{0}+\frac{x^{(n-1)}_{0,1}}{\tilde{w}^{(n-1)}_{0}x^{(n-1)}_{0,0}}\tilde{e}^{(n-1)}_{0}-\frac{1}{\tilde{w}^{(n-1)}_{0}}\tilde{e}^{(n-1)}_{1}$.
Then $v_{n-1}$ is a nonzero element in the space $\tilde{\mathcal{H}}_{n-1}$, which satisfies that $X_{n-1,n-1}(v_{n-1})+X_{n-1,n}(\tilde{e}^{(n)}_{0})=0$. Now suppose that we have found $v_{n-1},\cdots,v_{n-k}$ that satisfy $v_{n-i}=\sum\limits_{j=0}^{i}v^{(n-i)}_{j}\tilde{e}^{(n-i)}_{j}$ for $1\le i\le k$ and
\begin{equation*}
\setlength\abovedisplayskip{3pt}
\setlength\belowdisplayskip{3pt}
\begin{pmatrix}
X_{n-k,n-k}&X_{n-k,n-k+1}&X_{n-k,n-k+2}&\cdots&X_{n-k,n}\\
0&X_{n-k+1,n-k+1}&X_{n-k+1,n-k+2}&\cdots&X_{n-k+1,n}\\
\vdots&\ddots&\ddots&\ddots&\vdots\\
0&\cdots&0&X_{n-1,n-1}&X_{n-1,n}\\
0&\cdots&\cdots&0&X_{n,n}
\end{pmatrix}
\begin{pmatrix}
v_{n-k}\\
v_{n-k+1}\\
\vdots\\
v_{n-1}\\
\tilde{e}^{(n)}_{0}
\end{pmatrix}
=\begin{pmatrix}
0\\
0\\
\vdots\\
0\\
0
\end{pmatrix}.
\end{equation*}
We will prove that there exists an element $v_{n-k-1} $ in the space $\tilde{\mathcal{H}}_{n-k-1}$ of the form $v_{n-k-1}=\sum\limits_{j=0}^{k+1}v^{(n-k-1)}_{j}\tilde{e}^{(n-k-1)}_{j}$ such that the following equation holds:
\begin{equation*}
\setlength\abovedisplayskip{3pt}
\setlength\belowdisplayskip{3pt}
X_{n-k-1,n-k-1}(v_{n-k-1})=-\big(\sum\limits_{i=n-k}^{n-1}X_{n-k-1,i}(v_{i})+X_{n-k-1,n}(\tilde{e}^{(n)}_{0})\big).
\end{equation*}
Note that $X_{n-k-1,n}(\tilde{e}^{(n)}_{0})=\sum\limits_{m=0}^{k+1}x^{(n-k-1,n)}_{m,0}e^{(n-k-1)}_{m}$ and for $n-k\le i\le n-1$,
\begin{equation*}
\setlength\abovedisplayskip{3pt}
\setlength\belowdisplayskip{3pt}
X_{n-k-1,i}(v_{i})=X_{n-k-1,i}\big(\sum\limits_{j=0}^{n-i}v^{(i)}_{j}\tilde{e}^{(i)}_{j}\big)
=\sum\limits_{j=0}^{n-i}v^{(i)}_{j}\big(\sum\limits_{m=0}^{i-n+k+1+j}x^{(n-k-1,i)}_{m,j}e^{(n-k-1)}_{m}\big),
\end{equation*}
we get that $X_{n-k-1,i}(v_{i})=\sum\limits_{m=0}^{k+1}P^{(n-k-1,i)}_{m}e^{(n-k-1)}_{m}$ for $n-k\le i\le n-1$ and constants $P^{(n-k-1,i)}_{m}$.
Then we have that $-\big(\sum\limits_{i=n-k}^{n-1}X_{n-k-1,i}(v_{i})+X_{n-k-1,n}(\tilde{e}^{(n)}_{0})\big)=\sum\limits_{m=0}^{k+1}\tilde{P}_{m}e^{(n-k-1)}_{m}$
for some constants $\tilde{P}_{m}$. Since $X_{n-k-1,n-k-1}(\tilde{e}^{(n-k-1)}_{l})=\sum\limits_{m=0}^{l}x^{(n-k-1)}_{m,l}e^{(n-k-1)}_{m}$ for $l\ge0$, we have that $e^{(n-k-1)}_{0}=X_{n-k-1,n-k-1}(\frac{1}{x^{(n-k-1)}_{0,0}}\tilde{e}^{(n-k-1)}_{0})$ and
\begin{equation*}
\setlength\abovedisplayskip{3pt}
\setlength\belowdisplayskip{3pt}
e^{(n-k-1)}_{l}=X_{n-k-1,n-k-1}\big(\sum\limits_{m=0}^{l-1}\hat{x}^{(n-k-1)}_{m,l}\tilde{e}^{(n-k-1)}_{m}+\frac{1}{x^{(n-k-1)}_{l,l}}\tilde{e}^{(n-k-1)}_{l}\big)
\end{equation*}
for $l\le1$ and some constants $\hat{x}^{(n-k-1)}_{m,l}$. Therefore, we obtain that
\begin{align*}
\sum\limits_{m=0}^{k+1}\tilde{P}_{m}e^{(n-k-1)}_{m}
&=X_{n-k-1,n-k-1}(\frac{\tilde{P}_{0}}{x^{(n-k-1)}_{0,0}}\tilde{e}^{(n-k-1)}_{0})\\
&+\sum\limits_{m=1}^{k+1}\tilde{P}_{m}X_{n-k-1,n-k-1}\big(\sum\limits_{i=0}^{m-1}\hat{x}^{(n-k-1)}_{i,m}\tilde{e}^{(n-k-1)}_{i}+\frac{1}{x^{(n-k-1)}_{m,m}}\tilde{e}^{(n-k-1)}_{m}\big)\\
&=X_{n-k-1,n-k-1}(\frac{\tilde{P}_{0}}{x^{(n-k-1)}_{0,0}}\tilde{e}^{(n-k-1)}_{0})\\
&+X_{n-k-1,n-k-1}\big(\sum\limits_{m=1}^{k+1}\sum\limits_{i=0}^{m-1}\tilde{P}_{m}\hat{x}^{(n-k-1)}_{i,m}\tilde{e}^{(n-k-1)}_{i}\big)\\
&+X_{n-k-1,n-k-1}\big(\sum\limits_{m=1}^{k+1}\frac{\tilde{P}_{m}}{x^{(n-k-1)}_{m,m}}\tilde{e}^{(n-k-1)}_{m}\big)\\
&=X_{n-k-1,n-k-1}\big(\sum\limits_{m=0}^{k+1}Q_{m}\tilde{e}^{(n-k-1)}_{m}\big),
\end{align*}
for constants $Q_{m}$. Setting $v_{n-k-1}:=\sum\limits_{j=0}^{k+1}Q_{j}\tilde{e}^{(n-k-1)}_{j}$.
Based on the previous proof, we can find a nonzero vector $v$ that satisfies $X(v)=0$. However, this contradicts the fact that $X$ is injective. Therefore, $x^{(n)}_{0,0}\ne0$.

By analyzing the above process similarly for $Y$, we can also get $y^{(k)}_{0,0}\ne0$ for $1\le k\le n$. Thus, we have that
$0<\frac{\vert y^{(i)}_{0,0}\vert}{\Vert Y\Vert}\le\Big|\frac{\prod\limits_{k=0}^{l}\tilde{w}^{(i)}}{\prod\limits_{k=0}^{l}w^{(i)}}\Big|\le\frac{\Vert X\Vert}{\vert x^{(i)}_{0,0}\vert}<\infty$ for $l\ge0$ and $1\le i\le n$, which, by Shields's result in \cite{Shields}, means that $T_{i,i}\sim_{s}\tilde{T}_{i,i}$ for $1\le i\le n$.

Lastly, we define the operator $\tilde{X}_{i,i}:\mathcal{\tilde{H}}_{i}\rightarrow\mathcal{H}_{i}$, $1\le i\le n$, by $\tilde{X}_{i,i}(\tilde{e}^{(i)}_{0})=e^{(i)}_{0}$ and for $l\ge1$, $\tilde{X}_{i,i}(\tilde{e}^{(i)}_{l})=\frac{\prod\limits_{m=0}^{l-1}\tilde{w}^{(i)}_{m}}{\prod\limits_{m=0}^{l-1}w^{(i)}_{m}}e^{(i)}_{l}$.
Due to the similarity between $T_{k,k}$ and $\tilde{T}_{k,k}$, the operator $\tilde{X}_{k,k}$ is invertible for each $1\le k\le n$.
Set $\tilde{X}:=\tilde{X}_{1,1}\oplus\cdots\oplus\tilde{X}_{n,n}$.
Then, by a routine computation, the invertible operator $\tilde{X}$ satisfies $T\tilde{X}=\tilde{T}\tilde{X}$. This completes the proof.
\end{proof}

\textbf{Proof of Theorem \ref{theorem4.1}}: The conclusion can be proved directly together with Lemma \ref{lemma4} and Proposition \ref{proposition4.8}.

Finally, we apply Theorem \ref{theorem4.1} to Cowen-Douglas operators and multiplication operators acting on certain vector-valued reproducing kernel Hilbert spaces.
This allows us to provide many examples in which the similarity and quasi-similarity of high-index operators are identical.
For an introduction to Hilbert spaces containing analytic functions, the works \cite{Aronszajn1950,CS2,PR58} are good references.
%Although we have mentioned the Cowen-Douglas operators several times in previous texts, we have not yet defined them.
Below, we present a definition and point out the connection between these operators and the reproducing kernel theory.

\begin{defn}[\cite{CD}, Definition 1.2]
For a connected open subset $\Omega$ of $\mathbb{C}$ and a positive integer $n$, the class of Cowen-Douglas operators $\mathcal{B}_n(\Omega)$ is defined as follows:
$$\begin{array}{lll}{\mathcal{B}}_n(\Omega)=\{T\in \mathcal{B}(\mathcal{H}):
&(1)\,\,\Omega\subset \sigma(T):=\{w\in \mathbb{C}:T-w
\mbox{ is not invertible}\},\\
&(2)\,\,\mbox{dim}\ker(T-w)=n \text{ for all }w\in\Omega, \\
&(3)\,\,\bigvee_{w\in \Omega}\ker(T-w)=\mathcal{H},\text{ and}\\
&(4)\,\,\mbox{ran}(T-w)=\mathcal{H}\text{ for all }w\in\Omega\}.
\end{array}$$
\end{defn}

\begin{thm}[\cite{CD,CS2}]
An operator $T\in\mathcal{B}_n(\Omega)$ can be realized as the adjoint of the multiplication operator on a reproducing kernel Hilbert space of holomorphic $\mathbb{C}^n$-valued functions on $\Omega^*$ up to unitary equivalence, i.e. $T\sim_u(M_z^*,\mathcal{H}_K)$, $K:\Omega^*\times\Omega^*\rightarrow\mathcal{M}_n(\mathbb{C})$.
\end{thm}

The backward weighted shift operator is typically the Cowen-Douglas operator with index 1.
Cowen-Douglas operators have a rich and complicated structure. Even when the flag structure is added to these operators, many problems remain unknown, such as the classification of similarity and quasi-similarity.
Homogeneous operators in this class have good properties, thus their representations and unitary classifications are completely characterized by A. Kor\'{a}nyi and G. Misra in \cite{CM,KM,KMi09,KM1,KM2,Misra}.
An operator $T\in\mathcal{B}(\mathcal{H})$ with $\sigma(T)\subseteq\bar{\mathbb{D}}$ is said to be homogeneous if $\phi(T)$ is unitarily equivalent to $T$ for all $\phi\in$ M\"ob, where M\"ob denotes the group of all biholomorphic automorphisms of the unit disc $\mathbb{D}$.
In particular, using the curvature, a classical invariant of the line bundle, G. Misra in \cite{Misra} gave a concrete model for the homogeneous operator in $\mathcal{B}_1(\mathbb{D})$ as $T\sim_u(M_z^*,\mathcal{H}_{K^{(\lambda)}})$, $\lambda>0$, where $M^{*}_{z}$ is the adjoint of the multiplication operator on the reproducing kernel Hilbert space $\mathcal{H}_{K^{(\lambda)}}$ with the reproducing kernel $K^{(\lambda)}(z,w)=\frac{1}{(1-z\bar{w})^{\lambda}}$, $z,w\in\mathbb{D}$.
The results concerning homogeneous operators establish a deep connection between branches of operator theory, complex geometry, group representation theory, and reproducing kernel theory.
Furthermore, these operators facilitate the study of general Cowen-Douglas operators, and can be used as models and special cases to address various problems.
For instance, the irreducible homogeneous operators in $\mathcal{B}_2(\mathbb{D})$ can be represented by the homogeneous operators in $\mathcal{B}_1(\mathbb{D})$, which naturally have a flag structure.
More precisely, $T$ is a homogeneous operator in $\mathcal{FB}_2(\mathbb{D})$ if and only if $T\sim_u(M_z^*,\mathcal{H}_{K^{(\lambda,\lambda+2)}})$, where $\lambda>0$ and
\begin{equation*}
\setlength\abovedisplayskip{3pt}
\setlength\belowdisplayskip{3pt}
K^{(\lambda,\lambda+2)}(z,w)=\begin{pmatrix}K^{(\lambda)}(z,w)&\frac{\partial}{\partial\bar{w}}K^{(\lambda)}(z,w)\\
\frac{\partial}{\partial z}K^{(\lambda)}(z,w)&\frac{\partial^2}{\partial z\partial\bar{w}}K^{(\lambda)}(z,w)\end{pmatrix}
+\begin{pmatrix}0&0\\0&K^{(\lambda+2)}(z,w)\end{pmatrix}
\end{equation*}
for $z,w\in\mathbb{D}$.
The similarity equivalence problem of operators in this class have been partially inscribed in \cite{JJKM,JX,JJK,JJM,XJ}.
We then obtain a result for their quasi-similarity relation, finding that their similarity equivalence and quasi-similarity equivalence are the same.
This result is also a generalisation of Proposition 2.6 from M. Uchiyama's work in \cite{Uchiyama}.

\begin{example}\label{250724}
Let $T\sim_{u}(M^{*}_{z},\mathcal{H}_{K^{(\lambda,\lambda+2)}})$ and $\tilde{T}\sim_{u}(M^{*}_{z},\mathcal{H}_{\tilde{K}^{(\tilde{\lambda},\tilde{\lambda}+2)}})$. Then $T\sim_{q.s}\tilde{T}$ if and only if $T\sim_{s}\tilde{T}$.
\end{example}

According to Example \ref{250724}, we can provide an example involving the $\mathbb{C}^3$-valued reproducing kernel Hilbert space.
This result is also valid for the more general case of $\mathbb{C}^n$-valued reproducing kernel Hilbert spaces.
However, due to the overly complicated form of the general case, we will only offer a detailed introduction to the $n=3$ case.

Let $K_{i}(z,w)=\sum\limits_{n=0}^{\infty}a_{n}^{(i)}(z\bar{w})^{n},\ a_{n}^{(i)}>0,\ z,w\in\mathbb{D},\ i=1,2,3$
and
\begin{equation*}
\setlength\abovedisplayskip{3pt}
\setlength\belowdisplayskip{3pt}
K(z,w):=\left(\begin{smallmatrix}
K_{1}(z,w)&-\frac{\partial}{\partial\bar{w}}K_{1}(z,w)&\Diamond\\
-\frac{\partial}{\partial z}K_{1}(z,w)&\frac{\partial^{2}}{\partial z\partial\bar{w}}K_{1}(z,w)&\Diamond\Diamond\\
\Diamond\Diamond\Diamond&\Diamond\Diamond\Diamond\Diamond
&\Diamond\Diamond\Diamond\Diamond\Diamond
\end{smallmatrix}\right)
+\left(\begin{smallmatrix}
0&0&0\\
0&K_{2}(z,w)&-\frac{\partial}{\partial\bar{w}}K_{2}(z,w)\\
0&-\frac{\partial}{\partial z}K_{2}(z,w)&\frac{\partial^{2}}{\partial z\partial\bar{w}}K_{2}(z,w)
\end{smallmatrix}\right)
+\left(\begin{smallmatrix}
0&0&0\\
0&0&0\\
0&0&K_{3}(z,w)
\end{smallmatrix}\right),
\end{equation*}
where
{\small
\begin{align*}
\Diamond:=&\frac{1}{2}\frac{\partial^{2}}{\partial\bar{w}^{2}}K_{1}(z,w)-\frac{\partial}{\partial\bar{w}}(\phi(\bar{w})K_{1}(z,w)),\\
\Diamond\Diamond:=&-\frac{1}{2}\frac{\partial^{3}}{\partial z\partial\bar{w}^{2}}K_{1}(z,w)+\frac{\partial^{2}}{\partial z\partial\bar{w}}(\phi(\bar{w})K_{1}(z,w)),\\
\Diamond\Diamond\Diamond:=&\frac{1}{2}\frac{\partial^{2}}{\partial z^{2}}K_{1}(z,w)-\frac{\partial}{\partial z}(\overline{\phi(\bar{z})}K_{1}(z,w)),\\
\Diamond\Diamond\Diamond\Diamond:=&-\frac{1}{2}\frac{\partial^{3}}{\partial z^{2}\partial\bar{w}}K_{1}(z,w)+\frac{\partial^{2}}{\partial z\partial\bar{w}}(\overline{\phi(\bar{z})}K_{1}(z,w)),\\
\Diamond\Diamond\Diamond\Diamond\Diamond:=&\frac{1}{4}\frac{\partial^{4}}{\partial z^{2}\partial\bar{w}^{2}}K_{1}(z,w)-\frac{1}{2}\frac{\partial^{3}}{\partial z\partial\bar{w}^{2}}(\overline{\phi(\bar{z})}K_{1}(z,w))-\frac{1}{2}\frac{\partial^{3}}{\partial z^{2}\partial\bar{w}}(\phi(\bar{w})K_{1}(z,w))\\
&+\frac{\partial^{2}}{\partial z\partial\bar{w}}(\overline{\phi(\bar{z})}\phi(\bar{w})K_{1}(z,w))
\end{align*}
}
for $z,w\in\mathbb{D}$ and the function $\phi$ is holomorphic in the multiplier algebra,
\begin{equation*}
\setlength\abovedisplayskip{3pt}
\setlength\belowdisplayskip{3pt}
\mbox{Mult}(\mathcal{H}_{K_1}):=\{\psi\in\mbox{Hol}(\mathbb{D}):\psi f\in\mathcal{H}_{K_1}\ \mbox{for\ all}\ f\in\mathcal{H}_{K_1}\},
\end{equation*}
of the reproducing kernel Hilbert space $\mathcal{H}_{K_1}$ determined by $K_1$.
The binary function $\tilde{K}:\mathbb{D}\times\mathbb{D}\rightarrow\mathcal{M}_3(\mathbb{C})$ has a similar form to $K$ involving $\tilde{K}_{i}(z,w)=\sum\limits_{n=0}^{\infty}\tilde{a}_{n}^{(i)}(z\bar{w})^{n}$ and $\tilde{\phi}\in\mbox{Mult}(\mathcal{H}_{\tilde{K}_1})$, $\tilde{a}_{n}^{(i)}>0,\ z,w\in\mathbb{D},\ i=1,2,3$.
The next example will utilize these notations.
\begin{example}\label{corollary4.9}
Let $T=(M_{z}^{*},\mathcal{H}_{K})$ and $\tilde{T}=(M_{z}^{*},\mathcal{H}_{\tilde{K}})$.
%, where
%\begin{equation*}
%\setlength\abovedisplayskip{3pt}
%\setlength\belowdisplayskip{3pt}
%K_{i}(z,w)=\sum\limits_{n=0}^{\infty}a_{n}^{(i)}(z\bar{w})^{n},\ \tilde{K}_{i}(z,w)=\sum\limits_{n=0}^{\infty}\tilde{a}_{n}^{(i)}(z\bar{w})^{n},\ i=1,2,3
%\end{equation*}
%and $\phi,\tilde{\phi}$ are some holomorphic functions in the multiplier algebras $\mbox{Mult}(\mathcal{H}_{1}),\mbox{Mult}(\tilde{\mathcal{H}}_{1})$, respectively.
Then $T\sim_{q.s}\tilde{T}$ if and only if $T\sim_{s}\tilde{T}$.
\end{example}
\begin{proof}
Let
\begin{equation*}
\setlength\abovedisplayskip{3pt}
\setlength\belowdisplayskip{3pt}
\gamma_{1}(w)=\left(\begin{smallmatrix}K_{1}(\cdot,\bar{w})\\0\\0\end{smallmatrix}\right),
\gamma_{2}(w)=\left(\begin{smallmatrix}-\frac{\partial}{\partial w}K_{1}(\cdot,\bar{w})\\K_{2}(\cdot,\bar{w})\\0\end{smallmatrix}\right),
\gamma_{3}(w)=\left(\begin{smallmatrix}\frac{1}{2}\frac{\partial^{2}}{\partial w^{2}}K_{1}(\cdot,\bar{w})-\frac{\partial}{\partial w}(\phi(w)K_{1}(\cdot,\bar{w}))\\-\frac{\partial}{\partial w}K_{2}(\cdot,\bar{w})\\K_{3}(\cdot,\bar{w})\end{smallmatrix}\right).
\end{equation*}
Then $\gamma_i$, $i=1,2,3$, is a holomorphic function from $\mathbb{D}$ to $\mathcal{H}_{K_1}\oplus\mathcal{H}_{K_2}\oplus\mathcal{H}_{K_3}$ and they satisfy $K(z,w)=\big(\langle\gamma_{j}(\bar{w}),\gamma_{i}(\bar{z})\rangle\big)_{i,j=1}^{3}$. Similar constructions and verifications can be used to prove that $\tilde{H}_{K}$ also has a representation of this form.
These imply that $K$ and $\tilde{K}$ are non-negative definite kernels on $\mathbb{D}\times\mathbb{D}$,
and that $T=(M_{z}^{*},\mathcal{H}_{K})$ and $\tilde{T}=(M_{z}^{*},\mathcal{H}_{\tilde{K}})$ with respect to $\mathcal{H}_{K_1}\oplus\mathcal{H}_{K_2}\oplus\mathcal{H}_{K_3},\mathcal{H}_{\tilde{K}_1}\oplus\mathcal{H}_{\tilde{K}_2}\oplus\mathcal{H}_{\tilde{K}_3}$ are
$$T=\left(\begin{smallmatrix}T_{1,1}&T_{1,2}&T_{1,3}\\0&T_{2,2}&T_{2,3}\\0&0&T_{3,3}\end{smallmatrix}\right)
,\ \tilde{T}=\left(\begin{smallmatrix}\tilde{T}_{1,1}&\tilde{T}_{1,2}&\tilde{T}_{1,3}\\0&\tilde{T}_{2,2}&\tilde{T}_{2,3}\\0&0&\tilde{T}_{3,3}\end{smallmatrix}\right),\
\mbox{respectively}.$$

A routine verification shows that $(T-w)\gamma_{j}(w)=0$, $j=1,2,3$. More specifically, we have
$T_{i,i+1}K_{i+1}(\cdot,\bar{w})=K_{i}(\cdot,\bar{w}),\ i=1,2\ \,\text{and}\ \,T_{1,3}K_{3}(\cdot,\bar{w})=\phi(w)K_{1}(\cdot,\bar{w}),$
which means that $T_{i,i}T_{i,i+1}=T_{i,i+1}T_{i+1,i+1}\ \,\text{and}\ \,T_{1,3}=\phi(T_{1,1})T_{1,2}T_{2,3}.$
The same relationship applies to $\tilde{T}$ and $\tilde{K}$. Thus, both $T$ and $\tilde{T}$ possess flag structure and satisfy Condition (A).
Since $\{e_{l}^{(i)}:=\sqrt{a_{l}^{(i)}}z^{l}\}_{l=0}^{\infty}$ and $\{\tilde{e}_{l}^{(i)}:=\sqrt{\tilde{a}_{l}^{(i)}}z^{l}\}_{l=0}^{\infty}$ are the orthonormal bases of $\mathcal{H}_{K_i}$ and $\tilde{\mathcal{H}}_{\tilde{K}_i}$, respectively, we have
\begin{equation*}
\setlength\abovedisplayskip{3pt}
\setlength\belowdisplayskip{3pt}
T_{i,i+1}(e_{l}^{(i+1)})=\sqrt\frac{a_{l}^{(i)}}{a_{l}^{(i+1)}}e_{l}^{(i)},\ \tilde{T}_{i,i+1}(\tilde{e}_{l}^{(i+1)})=\sqrt\frac{\tilde{a}_{l}^{(i)}}{\tilde{a}_{l}^{(i+1)}}\tilde{e}_{l}^{(i)},\ l\ge0,\  i=1,2.
\end{equation*}
Therefore, the conclusion follows from Theorem \ref{theorem4.1}.
\end{proof}

\section{strong irreducibility for operators in $\mathcal{F}_{n}(\mathcal{A})$}\label{sec5}

%The above two results provide a specific form of intertwining between the two operators in $\mathcal{F}_{n}(\mathcal{A})$.
%Based on this characterisation,
%In this section, we explore the strong irreducibility of the operators in $\mathcal{F}_{n}(\mathcal{A})$ under quasi-similarity.
In this section, we explore the strong irreducibility of the operators in class $\mathcal{F}_{n}(\mathcal{A})$.
We knew that the class $\mathcal{F}_{n}(\mathcal{A})$ is exactly the Cowen-Douglas class with flag structure $\mathcal{FB}_{n}(\Omega)$ when $\mathcal{A}=\mathcal{B}_{1}(\Omega)$. Using the techniques from the spectral theory, C.L. Jiang and H. He showed that the strong irreducibility of operators in $\mathcal{B}_{n}(\Omega)$ is preserved under quasi-similairty.
Next we will use the language of operator matrices to prove that operators possessing Property (H) in the class $\mathcal{F}_{n}(\mathcal{A})$ preserve the strong irreducibility under quasi-similarity.

When $n=2$, the following equivalent condition for the strong irreducibility of an operator can be obtained.
This condition is also fundamental to describing the strong irreducibility of operators with a higher index.

\begin{lem}\label{lemma3.7}
Let $T=\left(\begin{smallmatrix}T_{1,1}&T_{1,2}\\0&T_{2,2}\end{smallmatrix}\right)\in\mathcal{F}_{2}(\mathcal{A})$, where $T_{i,i}\in(SI)$ and $\{T_{i,i}\}'$ is a  semi-simple commutative Banach algebra for each $i=1,2$. Then $T\in(SI)$ if and only if $T_{1,2}\notin\mbox{ran}\ \tau_{T_{1,1},T_{2,2}}$.
\end{lem}
\begin{proof}
We first assume that $T$ is strongly irreducible. If $T_{1,2}\in\mbox{ran}\ \tau_{T_{1,1},T_{2,2}}$, then there exists a bounded operator $X$ such that $T_{1,1}X-XT_{2,2}=T_{1,2}$.
Thus, we have an invertible operator $\left(\begin{smallmatrix}I&-X\\0&I\end{smallmatrix}\right)$ such that $T\sim_{s}T_{1,1}\oplus T_{2,2}$, which contradicts the strong irreducibility of $T$.

Conversely, suppose that $P$ is an idempotent operator in $\{T\}'$. Then the Proposition \ref{proposition3.6} tells us that $P$ is upper-triangular, denoted by $P=\left(\begin{smallmatrix}P_{1,1}&P_{1,2}\\0&P_{2,2}\end{smallmatrix}\right)$. Then we have
\begin{equation*}
\setlength\abovedisplayskip{3pt}
\setlength\belowdisplayskip{3pt}
\left(\begin{smallmatrix}T_{1,1}&T_{1,2}\\0&T_{2,2}\end{smallmatrix}\right)
\left(\begin{smallmatrix}P_{1,1}&P_{1,2}\\0&P_{2,2}\end{smallmatrix}\right)
=\left(\begin{smallmatrix}P_{1,1}&P_{1,2}\\0&P_{2,2}\end{smallmatrix}\right)
\left(\begin{smallmatrix}T_{1,1}&T_{1,2}\\0&T_{2,2}\end{smallmatrix}\right),
\end{equation*}
which implies that $T_{1,1}P_{1,1}=P_{1,1}T_{1,1}$ and $T_{2,2}P_{2,2}=P_{2,2}T_{2,2}$. Since both $T_{1,1}$ and $T_{2,2}$ are strongly irreducible, then we have $P_{i,i}=0$ or $P_{i,i}=I$ for $i=1,2$. Next, we analyze all possibilities in detail:

\textbf{Case 1:} We assume $P=\left(\begin{smallmatrix}I&P_{1,2}\\0&0\end{smallmatrix}\right)$. Then from $TP=PT$, we obtain that $T_{1,1}P_{1,2}-P_{1,2}T_{2,2}=T_{1,2}$. This contradicts the fact $T_{1,2}\notin\mbox{ran}\tau_{T_{1,1},T_{2,2}}$.

\textbf{Case 2:} We assume $P=\left(\begin{smallmatrix}0&P_{1,2}\\0&I\end{smallmatrix}\right)$. Then from $TP=PT$, we obtain that $T_{1,1}P_{1,2}-P_{1,2}T_{2,2}=-T_{1,2}$. The same reason leads to a contradiction.

\textbf{Case 3:} We assume $P=\left(\begin{smallmatrix}I&P_{1,2}\\0&I\end{smallmatrix}\right)$. Since $P$ is idempotent, then $P_{1,2}=0$. In this case, $P$ is a trivial idempotent in $\{T\}'$.

\textbf{Case 4:} We assume $P=\left(\begin{smallmatrix}0&P_{1,2}\\0&0\end{smallmatrix}\right)$. Since $P$ is idempotent, then $P_{1,2}=0$. In this case, $P$ is a trivial idempotent in $\{T\}'$.

In summary, we show that if $T_{1,2}\notin\mbox{ran}\tau_{T_{1,1},T_{2,2}}$, then the idempotent operator in $\{T\}'$ must be $I$ or $0$, which means that $T$ is strongly irreducible.
\end{proof}

When $n>2$, the operator's structure becomes extremely complicated.
It is straightforward to demonstrate that the aforementioned condition is no longer necessary for strong irreducibility when $n>2$.

\begin{example}\label{250723}
Let $T=\left(\begin{smallmatrix}T_{1,1}&I&0\\0&T_{1,1}&T_{1,2}\\0&0&T_{2,2}\end{smallmatrix}\right)\in\mathcal{F}_{3}(\mathcal{A})$, where $T_{i,i}\in(SI)$, $\{T_{i,i}\}'$ is semi-simple for $i=1,2$ and $T_{1,2}\in\mbox{ran}\,\tau_{T_{1,1},T_{2,2}}$ but $T_{1,2}\notin\mbox{ran}\,\tau_{T_{1,1},T_{2,2}}^{2}$.
Suppose that $P=\left(\begin{smallmatrix}P_{1,1}&P_{1,2}&P_{1,3}\\0&P_{2,2}&P_{2,3}\\0&0&P_{3,3}\end{smallmatrix}\right)$ is an idempotent operator in $\{T\}'$.
Since $\left(\begin{smallmatrix}T_{1,1}&I\\0&T_{1,1}\end{smallmatrix}\right)$ and $T_{2,2}$ are both strongly irreducible, the idempotent operator $P$ can only be of the form
$$\left(\begin{smallmatrix}I&0&0\\0&I&0\\0&0&I\end{smallmatrix}\right), \left(\begin{smallmatrix}0&0&0\\0&0&0\\0&0&0\end{smallmatrix}\right), \left(\begin{smallmatrix}I&0&P_{1,3}\\0&I&P_{2,3}\\0&0&0\end{smallmatrix}\right), \left(\begin{smallmatrix}0&0&P_{1,3}\\0&0&P_{2,3}\\0&0&I\end{smallmatrix}\right). $$
The last two forms of $P$ correspond to $T_{1,2}=\tau_{T_{1,1},T_{2,2}}(\tau_{T_{1,1},T_{2,2}}(-P_{1,3}))$ and $T_{1,2}=\tau_{T_{1,1},T_{2,2}}(\tau_{T_{1,1},T_{2,2}}(P_{1,3}))$, respectively. We have $P_{1,3}\ne0$ because $T_{1,2}\ne0$.
However, both cases contradict $T_{1,2}\notin\mbox{ran}\,\tau_{T_{1,1},T_{2,2}}^{2}$.
Therefore, these two forms of the idempotent operator $P$ do not belong to $\{T\}'$.
Thus, the only nonzero idempotent operator in $\{T\}'$ is the identity operator, meaning that $T$ is strongly irreducible.
\end{example}
In Example \ref{250723}, since $\{T_{1,1}\}'$ is semi-simple, $I\notin\mbox{ran}\,\tau_{T_{1,1}}$.
An analogous proof to that in Lemma \ref{lemma3.7} yields the result that this condition is sufficient for the operator to be strongly irreducible.

\begin{lem}\label{lemma3.10}
Let $T=\big((T_{i,j})\big)_{i,j=1}^{n}\in\mathcal{F}_{n}(\mathcal{A})$, where $T_{i,i}\in(SI)$ and $\{T_{i,i}\}'$ is a semi-simple commutative Banach algebra for each  $1\le i\le n$.
If $T_{i,i+1}\notin\mbox{ran}\ \tau_{T_{i,i},T_{i+1,i+1}}$ for $1\le i\le n-1$, then $T\in(SI)$.
\end{lem}

Next, we will present the main theorem of this section (see the following Theorem \ref{theorem5.5}).
This states that if a natural property is added to an operator in $\mathcal{F}_{n}(\mathcal{A})$, its strong irreducibility must be maintained up to quasi-similarity.
This property naturally exists among many operators, as detailed in Section 3 of \cite{JJK}.

\begin{defn}[\cite{JJK}, Definition 2.3]
Let $T_{1},T_{2}\in\mathcal{B}(\mathcal{H})$. We say that $T_{1}$ and $T_{2}$ satisfy \textbf{Property(H)} if the following condition holds: If $X$ is a bounded linear operator defined on $\mathcal{H}$ such that $T_{1}X=XT_{2}$  and $X=T_{1}Y-YT_{2}$ for some $Y\in\mathcal{B}(\mathcal{H})$, then $X=0$.
\end{defn}
Assuming that the operators satisfy Property (H), we can then identify a class of operators that satisfies the Question \ref{question1.1} posed by C.L. Jiang.

\begin{thm}\label{theorem5.5}
Let $T=\big((T_{i,j})\big)_{i,j=1}^{n},\tilde{T}=\big((\tilde{T}_{i,j})\big)_{i,j=1}^{n}\in\mathcal{F}_{n}(\mathcal{A})$, where $T_{i,i}$ (resp., $\tilde{T}_{i,i}$) is strongly irreducible and $\{T_{i,i}\}'$ (resp., $\{\tilde{T}_{i,i}\}'$) is a semi-simple commutative Banach algebra for each $1\le i\le n$. Suppose that $\tilde{T}$ satisfies Property (H). If $T\sim_{q.s}\tilde{T}$, then $T\in(SI)$.
\end{thm}
\begin{proof}
Suppose that $X=\big((X_{i,j})\big)_{i,j=1}^{n}$ and $Y=\big((Y_{i,j})\big)_{i,j=1}^{n}$ are two quasi-affine intertwiners such that $TX=X\tilde{T}$ and $YT=\tilde{T}Y$.
According to Proposition \ref{proposition3.7}, we know that $X$ and $Y$ are upper-triangular. From the relations $TX=X\tilde{T}$ and $YT=\tilde{T}Y$, we can deduce that, for $1\le i\le n$,
\begin{equation}\label{equation3.3}
\setlength\abovedisplayskip{3pt}
\setlength\belowdisplayskip{3pt}
T_{i,i}X_{i,i}=X_{i,i}\tilde{T}_{i,i},\ Y_{i,i}T_{i,i}=\tilde{T}_{i,i}Y_{i,i}.
\end{equation}
Next, we will show that $T$ also satisfies Property (H). For any $1\le i\le n-1$, we assume that $T_{i,i}Z_{i}=Z_{i}T_{i+1,i+1}$ and $T_{i,i}W_{i}-W_{i}T_{i+1,i+1}=Z_{i}$,
where $Z_{i}$ and $W_{i}$ are bounded operators. Then using the equation (\ref{equation3.3}), we have that
\begin{equation*}
\setlength\abovedisplayskip{3pt}
\setlength\belowdisplayskip{3pt}
\tilde{T}_{i,i}Y_{i,i}Z_{i}X_{i+1,i+1}=Y_{i,i}Z_{i}X_{i+1,i+1}\tilde{T}_{i+1,i+1}
\end{equation*}
and
\begin{equation*}
\setlength\abovedisplayskip{3pt}
\setlength\belowdisplayskip{3pt}
\tilde{T}_{i,i}Y_{i,i}W_{i}X_{i+1,i+1}-Y_{i,i}W_{i}X_{i+1,i+1}\tilde{T}_{i+1,i+1}=Y_{i,i}Z_{i}X_{i+1,i+1}.
\end{equation*}
Since $\tilde{T}$ satisfies Property (H), we have that $Y_{i,i}Z_{i}X_{i+1,i+1}=0$, $1\le i\le n-1$. In addition, the fact that both $Y_{i,i}$ and $X_{i+1,i+1}$ have dense ranges leads to the result $Z_{i}=0$. This shows that $T$ satisfies Property (H).
Since the nonzero operator $T_{i,i+1}$ is in the kernel of the Rosenblum operator $\tau_{T_{i,i},T_{i+1,i+1}}$ for $1\le i\le n-1$, we have $T_{i,i+1}\notin\mbox{ran}\ \tau_{T_{i,i},T_{i+1,i+1}}$, which implies that $T$ is strongly irreducible.
\end{proof}
According to Section 3 of \cite{JJK}, Property (H) can be formulated in terms of certain operator classes. This leads to the following two corollaries:
\begin{cor}
Let $T=\left(\begin{smallmatrix}T_{1,1}&T_{1,2}\\0&T_{2,2}\end{smallmatrix}\right),\tilde{T}=\left(\begin{smallmatrix}\tilde{T}_{1,1}&\tilde{T}_{1,2}\\0&\tilde{T}_{2,2}\end{smallmatrix}\right)\in\mathcal{F}_{2}(\mathcal{A})$, where $T_{i,i},\tilde{T}_{i,i}\in(SI)$ and $\{T_{i,i}\}',\{\tilde{T}_{i,i}\}'$ are semi-simple commutative Banach algebras for $i=1,2$, respectively. Suppose that $\ker\tau_{\tilde{T}_{2,2},\tilde{T}_{1,1}}\ne\{0\}$. If $T\sim_{q.s}\tilde{T}$, then $T\in(SI)$.
\end{cor}

In the case of the backward weighted shift operator, according to Proposition \ref{250721}, we have the following result.
\begin{cor}
Let $T=\left(\begin{smallmatrix}T_{1,1}&T_{1,2}\\0&T_{2,2}\end{smallmatrix}\right),\tilde{T}=\left(\begin{smallmatrix}\tilde{T}_{1,1}&\tilde{T}_{1,2}\\0&\tilde{T}_{2,2}\end{smallmatrix}\right)\in\mathcal{F}_{2}(\mathcal{A})$, where $T_{i,i},\tilde{T}_{i,i}$ are backward weighted shift operators with nonzero weighted sequences $\{w_{k}^{(i)}\}_{k\ge0}$ and $\{\tilde{w}_{k}^{(i)}\}_{k\ge0}$, respectively. Suppose that $\lim\limits_{l\rightarrow\infty}l\frac{\prod\limits_{k=0}^{l-1}\tilde{w}_{k}^{(2)}}{\prod\limits_{k=0}^{l-1}\tilde{w}_{k}^{(1)}}=\infty.$
If $T\sim_{q.s}\tilde{T}$, then $T\in(SI)$.
\end{cor}

%The following result suggests that, even without Property (H), many quasi-similarity transformations preserve the strong irreducibility of operators in $\mathcal{F}_{n}(\mathcal{A})$.

We will conclude this section with three examples.
They provide three pairs of operators that are quasi-similar equivalent, and all of them have strong irreducibility.
Most research on the strong irreducibility of operators has focused on Cowen-Douglas operators. The operators in these examples go beyond the Cowen-Douglas class.
\begin{example}
Let $S$ be the unilateral shift on $\mathcal{H}$. Let
\begin{equation*}
\setlength\abovedisplayskip{3pt}
\setlength\belowdisplayskip{3pt}
T=\left(\begin{smallmatrix}I+S^{*}&I\\0&I+S^{*}\end{smallmatrix}\right),\quad
\tilde{T}=\left(\begin{smallmatrix}I+S^{*}&I+S^{*}\\0&I+S^{*}\end{smallmatrix}\right),\quad
X=\left(\begin{smallmatrix}I&0\\0&I+S^{*}\end{smallmatrix}\right),\quad
Y=\left(\begin{smallmatrix}(I+S^{*})^{2}&0\\0&I+S^{*}\end{smallmatrix}\right)\ \mbox{on}\ \mathcal{H}\oplus\mathcal{H}.
\end{equation*}
Then we have $TX=X\tilde{T}$ and $YT=\tilde{T}Y$. By the Example 2.35 in \cite{JW}, we can know that $T\sim_{q.s}\tilde{T}$ and $T$ is strongly irreducible. If $(I+S^{*})Z-Z(I+S^{*})=I+S^{*}$, then $S^{*}Z-ZS^{*}=I+S^{*}$.
By comparing the image of the above equation acting on the orthonormal basis $\{e_{n}\}_{n\ge0}$ of $\mathcal{H}$, we obtain that
\begin{equation*}
\setlength\abovedisplayskip{3pt}
\setlength\belowdisplayskip{3pt}
Z(e_{n})=\sum\limits_{k=0}^{n-1}a_{n-k}e_{k}+(a_{0}+n)e_{n}+(n+1)e_{n+1},\ n\ge0,
\end{equation*}
where $a_{0},a_{1},\cdots,a_{n},\cdots$ are the numbers to be determined.
It follows that $\Vert Z(e_{n})\Vert^{2}=\sum\limits_{k=0}^{n-1}a_{n-k}^{2}+(a_{0}+n)^{2}+(n+1)^{2}\rightarrow\infty$ as $n\rightarrow\infty$.
Hence, $I+S^{*}\notin\mbox{ran}\ \tau_{I+S^{*}}$. It follows from Lemma \ref{lemma3.7} that $\tilde{T}$ is also strongly irreducible.
\end{example}

\begin{example}
Let $T\in\mathcal{A}$ be an arbitrary injective operator whose commutant algebra $\{T\}'$ is semi-simple. Define the two operators in $\mathcal{F}_{2}(\mathcal{A})$ in the following way
\begin{equation*}
\setlength\abovedisplayskip{3pt}
\setlength\belowdisplayskip{3pt}
T_{1}=\left(\begin{smallmatrix}T&I\\0&T\end{smallmatrix}\right),\  T_{2}=\left(\begin{smallmatrix}T&T\\0&T\end{smallmatrix}\right).
\end{equation*}
Set $X=\left(\begin{smallmatrix}I&0\\0&T\end{smallmatrix}\right)$ and $Y=\left(\begin{smallmatrix}T^{2}&0\\0&T\end{smallmatrix}\right)$.
Since $T$ is an injective operator in class $\mathcal{A}$, it is clear that $X$ and $Y$ are both quasi-affine.
By a routine calculation, we can see that $T_{1}X=XT_{2}$ and $YT_{1}=T_{2}Y$, which imply that $T_{1}$ and $T_{2}$ are quasi-similar.
Moreover, the fact that $\{T\}'$ is semi-simple illustrates that $T_{1}$ and $T_{2}$ are both strongly irreducible.
\end{example}

Lastly, the third example suggests that quasi-similarity transformations may still preserve the strong irreducibility of operators in $\mathcal{F}_{n}(\mathcal{A})$ even without Property (H).

\begin{example}
Let $T=\left(\begin{smallmatrix}T_{1,1}&Z\\0&T_{2,2}\end{smallmatrix}\right),\tilde{T}=\left(\begin{smallmatrix}T_{1,1}&T_{1,2}\\0&T_{2,2}\end{smallmatrix}\right)\in\mathcal{F}_{2}(\mathcal{A})$, where $Z$ is invertible, $T_{1,2}\not\in\mbox{ran}\ \tau_{T_{1,1},T_{2,2}}$ is injective, $i=1,2$.
Suppose that $T_{i,i}\in(SI)$ and $\{T_{i,i}\}'$ is semi-simple for $i=1,2$. Set
$X=\left(\begin{smallmatrix}T_{1,2}Z^{-1}&0\\0&I\end{smallmatrix}\right)\ \mbox{and}\
Y=\left(\begin{smallmatrix}I&0\\0&Z^{-1}T_{1,2}\end{smallmatrix}\right)$.
It is obvious that both $X$ and $Y$ are quasi-affine and satisfy $XT=\tilde{T}X$ and $TY=Y\tilde{T}$, which shows that $T\sim_{q.s}\tilde{T}$. Furthermore, $T$ is strongly irreducible. Otherwise, by Lemma \ref{lemma3.7}, $Z$ is in the range of the operator $\tau_{T_{1,1},T_{2,2}}$, then so is $T_{1,2}$. This would lead to a contradiction.
\end{example}

However, for the general case that does not possess Property (H), the following remark will show that the method of operator matrices does not work.
\begin{rem}
Let $T=\big((T_{i,j})\big)_{i,j=1}^{n},\tilde{T}=\big((\tilde{T}_{i,j})\big)_{i,j=1}^{n}\in\mathcal{F}_{n}(\mathcal{A})$ and $T\sim_{q,s}\tilde{T}$, where $T_{i,i},\tilde{T}_{i,i}\in(SI)$, $\{T_{i,i}\}',\{\tilde{T}_{i,i}\}'$ are semi-simple commutative Banach algebras for $1\le i\le n$ and $\tilde{T}_{i,i+1}\notin\mbox{ran}\ \tau_{\tilde{T}_{i,i},\tilde{T}_{i+1,i+1}}$ for $1\le i\le n-1$.
If $T$ is not strongly irreducible, then from Lemma \ref{lemma3.10}, there exists some operator $T_{p,p+1}$ such that $T_{p,p+1}\in\tau_{T_{p,p},T_{p+1,p+1}}$. This enables us to obtain that
\begin{equation*}
\setlength\abovedisplayskip{3pt}
\setlength\belowdisplayskip{3pt}
\left(\begin{smallmatrix}
T_{1,1}&T_{1,2}&\cdots&\cdots&\cdots&T_{1,n}\\
&\ddots&\ddots&\ddots&\ddots&\vdots\\
&&T_{p,p}&T_{p,p+1}&\cdots&T_{p,n}\\
&&&T_{p+1,p+1}&\cdots&T_{p+1,n}\\
&&&&\ddots&\vdots\\
&&&&&T_{n,n}
\end{smallmatrix}\right)
\sim_{s}
\left(\begin{smallmatrix}
T_{1,1}&T_{1,2}&\cdots&\cdots&\cdots&*\\
&\ddots&\ddots&\ddots&\ddots&\vdots\\
&&T_{p,p}&0&\cdots&*\\
&&&T_{p+1,p+1}&\cdots&*\\
&&&&\ddots&\vdots\\
&&&&&T_{n,n}
\end{smallmatrix}\right)\triangleq\hat{T}.
\end{equation*}
Furthermore, it follows from $T\sim_{q.s}\tilde{T}$ that $\tilde{T}\sim_{q.s}\hat{T}$.
\begin{comment}
\begin{equation*}
\setlength\abovedisplayskip{3pt}
\setlength\belowdisplayskip{3pt}
\left(\begin{smallmatrix}
\tilde{T}_{1,1}&\tilde{T}_{1,2}&\cdots&\cdots&\cdots&\tilde{T}_{1,n-1}&\tilde{T}_{1,n}\\
&\ddots&\ddots&\ddots&\ddots&\ddots&\vdots\\
&&\tilde{T}_{i,i}&\tilde{T}_{i,i+1}&\cdots&\cdots&\vdots\\
&&&\tilde{T}_{i+1,i+1}&\cdots&\cdots&\vdots\\
&&&&\ddots&\ddots&\ddots
\end{smallmatrix}\right)
\sim_{q.s}
\left(\begin{smallmatrix}
T_{1,1}&T_{1,2}&\cdots&\cdots&\cdots&T_{1,n-1}&T_{1,n}\\
&\ddots&\ddots&\ddots&\ddots&\ddots&\vdots\\
&&T_{i,i}&0&\cdots&\cdots&\vdots\\
&&&T_{i+1,i+1}&\cdots&\cdots&\vdots\\
&&&&\ddots&\ddots&\ddots
\end{smallmatrix}\right).
\end{equation*}
\end{comment}
Suppose that $X=\big((X_{i,j})\big)_{i,j=1}^{n}$ and $Y=\big((Y_{i,j})\big)_{i,j=1}^{n}$ are two quasi-affine intertwining operators such that $\tilde{T}X=X\hat{T}$ and $Y\tilde{T}=\hat{T}Y$.
Then according to the Proposition \ref{proposition3.7}, we have $X$ and $Y$ are upper-triangular.
Moreover, the relation $\tilde{T}X=X\hat{T}$ and %$Y\tilde{T}=\hat{T}Y$
implies that
\begin{equation*}
\setlength\abovedisplayskip{3pt}
\setlength\belowdisplayskip{3pt}
%\begin{aligned}
%\tilde{T}_{p+1,p+1}X_{p+1,p+1}&=X_{p+1,p+1}T_{p+1,p+1},\\
\tilde{T}_{p,p}X_{p,p+1}-X_{p,p+1}T_{p+1,p+1}=-\tilde{T}_{p,p+1}X_{p+1,p+1}.
%T_{p+1,p+1}Y_{p+1,p+1}&=Y_{p+1,p+1}\tilde{T}_{p+1,p+1}.
%\end{aligned}
\end{equation*}
%Thus we have
%\begin{equation}\label{equation3.4}
%\setlength\abovedisplayskip{3pt}
%\setlength\belowdisplayskip{3pt}
%\tilde{T}_{p,p}X_{p,p+1}Y_{p+1,p+1}-X_{p,p+1}Y_{p+1,p+1}\tilde{T}_{p+1,p+1}=-\tilde{T}_{p,p+1}X_{p+1,p+1}Y_{p+1,p+1}.
%\end{equation}
%Denote the operators $X_{p,p+1}Y_{p+1,p+1}$ and $X_{p+1,p+1}Y_{p+1,p+1}$ as $W$ and $Z$, respectively, and then note that $Z$ has dense range.
Consequently, the above equation illustrates that for any $y:=X_{p+1,p+1}(x)\in\mbox{ran}\ X_{p+1,p+1}$, a dense linear space of the space $\tilde{\mathcal{H}}_{p+1}$, we have
\begin{equation*}
	\setlength\abovedisplayskip{3pt}
	\setlength\belowdisplayskip{3pt}
\big(\tilde{T}_{p,p}X_{p,p+1}-X_{p,p+1}\tilde{T}_{p+1,p+1}\big)(x)=-\tilde{T}_{p,p+1}(y).
\end{equation*}
Subsequently, for any $y\in\tilde{\mathcal{H}}_{p+1}$, there exists a sequence $\{y_{k}:=X_{p+1,p+1}(x_{k})\}\subseteq\mbox{ran}\ X_{p+1,p+1}$ such that $y=\lim\limits_{k\rightarrow\infty}y_{k}$, and thus
\begin{equation*}
\setlength\abovedisplayskip{3pt}
\setlength\belowdisplayskip{3pt}
-\tilde{T}_{p,p+1}(y)
=-\lim\limits_{k\rightarrow\infty}\tilde{T}_{p,p+1}X_{p+1,p+1}(x_{k})
=\lim\limits_{k\rightarrow\infty}\big(\tilde{T}_{p,p}X_{p,p+1}-X_{p,p+1}\tilde{T}_{p+1,p+1}\big)(x_{k}).
%=\big(\tilde{T}_{p,p}W-W\tilde{T}_{p+1,p+1}\big)(x).
\end{equation*}

If $\{x_{k}\}$ is a Cauchy sequence and if there exists a bounded linear operator $Z$ such that $Z(y):=\lim\limits_{k\rightarrow\infty}X_{p,p+1}(x_{k})$, then we obtain that $\tilde{T}_{p,p}Z-Z\tilde{T}_{p+1,p+1}=-\tilde{T}_{p,p+1}$, which will make a contradiction. Thus $T\in(SI)$.

However, these two conditions does not hold in general. First of all, the condition ``$\{x_{k}\}$ is a Cauchy sequence" implies that $y\in\mbox{ran}\ X_{p+1,p+1}$, which means that $X_{p+1,p+1}$ has closed range.
Even if the first condition is true, whether the operator $Z$ is well-defined will be affected by the injectivity of operator $X_{p+1,p+1}$. Moreover, if both of the conditions hold true, we will obtain that $X_{p+1,p+1}$ is an invertible operator.

In addition, we can also understand the difficulty of using operator matrices to deal with the quasi-similarity problem compared to the similarity problem. If $X=\big((X_{i,j})\big)_{i,j=1}^{n}$ is an  invertible intertwining operator, then we can obtain that $X_{i,i}$, $1\le i\le n$, is invertible; and if $X=\big((X_{i,j})\big)_{i,j=1}^{n}$ is a quasi-affine intertwining operator, then we can only get that $X_{1,1}$ is quasi-affine.

%Since $X$ is injective and $T_{1,1}X_{1,1}=X_{1,1}\tilde{T}_{1,1}$, then $X_{1,1}$ is injective and has dense range. Then following the assumption that $X_{1,1}$ has closed range, we obtain that $X_{1,1}$ is invertible. Thus $\tilde{T}_{1,2}$ is in the range of the operator $\tau_{\tilde{T}_{1,1},\tilde{T}_{2,2}}$, which makes a contradiction with the hypothesis.
\end{rem}

\section{Appendix}
Throughout the article, we realized that the fact that the two quasi-affine intertwining operators between any two operators in $\mathcal{F}_{n}(\mathcal{A})$ are both in the form of upper-triangular is the most basic and important premise work. For the sake of completeness of the paper, we present the specific proof details in this section.
%In this section, for the sake of the completeness of the paper, we characterize the structure of the operator in the commutant algebra of operator in $\mathcal{F}_{n}(\mathcal{A})$.
%Moreover, we describe that any two quasi-affines intertwining operators between two different operators in $\mathcal{F}_{n}(\mathcal{A})$ is upper-triangular form. %Therefore, to begin with, we will investigate the properties and composition of $\{T\}'$ for $T\in\mathcal{F}_{n}(\mathcal{A})$.

The quasi-nilpotent operator plays a key role in this process. Next, we will use the following property:
\begin{thm}[\cite{K}]\label{lemma2.1}
	Let $P$ and $T$ be two bounded linear operators. If $P\in\ker\tau_{T}\cap\mbox{ran}\,\tau_{T}$, then $P$ is a quasi-nilpotent.
\end{thm}
As stated in Section \ref{sec2}, if the commutant algebra of an operator is semi-simple, then there are no nonzero quasi-nilpotent operators in that algebra.
An immediate consequence of this theorem is the following evident fact, which we state as a corollary.

\begin{cor}\label{corollary3.5}
	Let $T$ be a bounded linear operator. If $\{T\}'$ is semi-simple, then $\ker\tau_{T}\cap\mbox{ran}\,\tau_{T}=\{0\}$.
\end{cor}

The properties of operators with index one are often easy to determine.
Adding a flag structure to the operator matrix generalises the results of many low-index operators to the high-index case. The following proposition shows that the commutant algebras of operators with flag structure possess a nice property. More specifically, the property of the commutant algebra of the operator's atoms can be effectively transmitted to the entire operator under the flag structure.

\begin{prop}\label{proposition3.6}
	Let $T=\big((T_{i,j})\big)_{i,j=1}^{n}\in\mathcal{F}_{n}(\mathcal{A})$ that satisfies $\{T_{i,i}\}'$ is a semi-simple commutative Banach algebra for $1\le i\le n$. If $X$ is a bounded linear operator such that $XT=TX$, then the operator matrix of $X$ is of upper-triangular form.
\end{prop}
\begin{proof}
	We first prove the case of $n=2$. Let $X=\left(\begin{smallmatrix}X_{1,1}&X_{1,2}\\X_{2,1}&X_{2,2}\end{smallmatrix}\right)$ be an intertwining operator such that
	\begin{equation*}
		\setlength\abovedisplayskip{3pt}
		\setlength\belowdisplayskip{3pt}
		\left(\begin{smallmatrix}X_{1,1}&X_{1,2}\\X_{2,1}&X_{2,2}\end{smallmatrix}\right)
		\left(\begin{smallmatrix}T_{1,1}&T_{1,2}\\0&T_{2,2}\end{smallmatrix}\right)=
		\left(\begin{smallmatrix}T_{1,1}&T_{1,2}\\0&T_{2,2}\end{smallmatrix}\right)
		\left(\begin{smallmatrix}X_{1,1}&X_{1,2}\\X_{2,1}&X_{2,2}\end{smallmatrix}\right).
	\end{equation*}
	Then it follows that $X_{2,1}T_{1,1}=T_{2,2}X_{2,1}$ and $X_{2,1}T_{1,2}+X_{2,2}T_{2,2}=T_{2,2}X_{2,2}$,
	which implies that the operator $X_{2,1}T_{1,2}$ is in $\ker\tau_{T_{2,2}}\cap\mbox{ran}\,\tau_{T_{2,2}}$. Thus $X_{2,1}T_{1,2}=0$ by Corollary \ref{corollary3.5}. Since $T_{1,2}$ has dense range, $X_{2,1}=0$. Otherwise, $X_{2,1}$ also has dense range, which contradicts the equality $X_{2,1}T_{1,2}=0$.
	
	Next, we consider the general case of $n$. Let $X=\big((X_{i,j})\big)_{i,j=1}^{n}$ be an intertwining operator such that $XT=TX$. Then we have the following claim:
	\begin{claim}
		$X_{n,j}=0$ for $1\le j\le n-1$.
	\end{claim}
	Following from the relation $XT=TX$, we could find that
	\begin{equation}\label{Equation3.3}
		\setlength\abovedisplayskip{3pt}
		\setlength\belowdisplayskip{3pt}
		X_{n,1}T_{1,j}+X_{n,2}T_{2,j}+\cdots+X_{n,j}T_{j,j}=T_{n,n}X_{n,j},\ 1\le j\le n.
	\end{equation}
	Then from the two equations obtained by setting $j = 1,2$ in Equation (\ref{Equation3.3}), we obtain that $X_{n,1}T_{1,2}T_{2,3}\cdots T_{n-1,n}\in \ker\tau_{T_{n,n}}\cap\mbox{ran}\,\tau_{T_{n,n}}$.
	Subsequently, the same reason leads to the conclusion $X_{n,1}=0$. Then the two equations obtained by setting $j = 2,3$ in Equation (\ref{Equation3.3}) shows that $X_{n,2}T_{2,3}\cdots T_{n-1,n}\in \ker\tau_{T_{n,n}}\cap\mbox{ran}\,\tau_{T_{n,n}}$. Thus $X_{n,2}=0$. This claim will be proved by repeating the above progress step by step.
	
	Now, let us write the two operators $T$ and $X$ in the form of $2\times 2$ block matrices:
	\begin{equation*}
		\setlength\abovedisplayskip{3pt}
		\setlength\belowdisplayskip{3pt}
		T=\left(\begin{smallmatrix}T_{n-1}&T_{(n-1)\times 1}\\0&T_{n,n}\end{smallmatrix}\right),\
		X=\left(\begin{smallmatrix}X_{n-1}&X_{(n-1)\times 1}\\0&X_{n,n}\end{smallmatrix}\right),
	\end{equation*}
	where $X_{n-1}=\big((X_{i,j})\big)_{i,j=1}^{n-1}$ and $X_{(n-1)\times 1}=(X_{1,n},\cdots,X_{n-1,n})\textsuperscript{T}$.
	Following the relation $XT=TX$, we get
	\begin{equation*}
		\setlength\abovedisplayskip{3pt}
		\setlength\belowdisplayskip{3pt}
		\left(\begin{smallmatrix}X_{n-1}&X_{(n-1)\times 1}\\0&X_{n,n}\end{smallmatrix}\right)
		\left(\begin{smallmatrix}T_{n-1}&T_{(n-1)\times 1}\\0&T_{n,n}\end{smallmatrix}\right)
		=\left(\begin{smallmatrix}T_{n-1}&T_{(n-1)\times 1}\\0&T_{n,n}\end{smallmatrix}\right)
		\left(\begin{smallmatrix}X_{n-1}&X_{(n-1)\times 1}\\0&X_{n,n}\end{smallmatrix}\right),
	\end{equation*}
	which implies that $X_{n-1}T_{n-1}=T_{n-1}X_{n-1}$.
	By an analogous proof, we can obtain that $X_{i,j}=0$ for all $1\le j<i\le n-1$, which shows that $X$ has an upper-triangular form.
\end{proof}

Moreover, when we consider two quasi-affine intertwining operators between two different operators in $\mathcal{F}_{n}(\mathcal{A})$, we find that they are both of upper-triangular form. This generalizes the result of the fourth author, together with C.L. Jiang, D.K. Keshari and G. Misra, in paper \cite{JJKM} to the quasi-affinity relation between operators that are broader than similar equivalence, while also expanding the operator class.

\begin{prop}\label{proposition3.7}
	Let $T=\big((T_{i,j})\big)_{i,j=1}^{n},\tilde{T}=\big((\tilde{T}_{i,j})\big)_{i,j=1}^{n}\in\mathcal{F}_{n}(\mathcal{A})$ satisfy the condition that $\{T_{i,i}\}'$ (resp., $\{\tilde{T}_{i,i}\}'$) is semi-simple commutative Banach algebra for $1\le i\le n$. If $X$ and $Y$ are two quasi-affine intertwiners such that $XT=\tilde{T}X$ and $TY=Y\tilde{T}$, then the operator matrices of $X$ and $Y$ are both of upper-triangular form.
\end{prop}

%To better understand the Proposition \ref{proposition3.7}, we have its following detailed proof.
\begin{proof}
	We first prove the case of $n=2$. Let $X=\left(\begin{smallmatrix}X_{1,1}&X_{1,2}\\X_{2,1}&X_{2,2}\end{smallmatrix}\right)$ and $Y=\left(\begin{smallmatrix}Y_{1,1}&Y_{1,2}\\Y_{2,1}&Y_{2,2}\end{smallmatrix}\right)$ be two quasi-affine intertwining operators such that
	\begin{equation*}
		\setlength\abovedisplayskip{3pt}
		\setlength\belowdisplayskip{3pt}
		\left(\begin{smallmatrix}X_{1,1}&X_{1,2}\\X_{2,1}&X_{2,2}\end{smallmatrix}\right)
		\left(\begin{smallmatrix}T_{1,1}&T_{1,2}\\0&T_{2,2}\end{smallmatrix}\right)=
		\left(\begin{smallmatrix}\tilde{T}_{1,1}&\tilde{T}_{1,2}\\0&\tilde{T}_{2,2}\end{smallmatrix}\right)
		\left(\begin{smallmatrix}X_{1,1}&X_{1,2}\\X_{2,1}&X_{2,2}\end{smallmatrix}\right),
	\end{equation*}
	and
	\begin{equation*}
		\setlength\abovedisplayskip{3pt}
		\setlength\belowdisplayskip{3pt}
		\left(\begin{smallmatrix}T_{1,1}&T_{1,2}\\0&T_{2,2}\end{smallmatrix}\right)
		\left(\begin{smallmatrix}Y_{1,1}&Y_{1,2}\\Y_{2,1}&Y_{2,2}\end{smallmatrix}\right)
		=\left(\begin{smallmatrix}Y_{1,1}&Y_{1,2}\\Y_{2,1}&Y_{2,2}\end{smallmatrix}\right)
		\left(\begin{smallmatrix}\tilde{T}_{1,1}&\tilde{T}_{1,2}\\0&\tilde{T}_{2,2}\end{smallmatrix}\right).
	\end{equation*}
	Then it follows that $X_{2,1}T_{1,1}=\tilde{T}_{2,2}X_{2,1},\ X_{2,1}T_{1,2}+X_{2,2}T_{2,2}=\tilde{T}_{2,2}X_{2,2}$ and $ T_{2,2}Y_{2,1}=Y_{2,1}\tilde{T}_{1,1}$.
	Furthermore, we have
	$X_{2,1}T_{1,2}Y_{2,1}\tilde{T}_{1,2}+X_{2,2}Y_{2,1}\tilde{T}_{1,1}\tilde{T}_{1,2}=\tilde{T}_{2,2}X_{2,2}Y_{2,1}\tilde{T}_{1,2}$.
	Since $\tilde{T}_{1,1}\tilde{T}_{1,2}=\tilde{T}_{1,2}\tilde{T}_{2,2}$, we obtain that $\tilde{T}_{2,2}X_{2,2}Y_{2,1}\tilde{T}_{1,2}-X_{2,2}Y_{2,1}\tilde{T}_{1,2}\tilde{T}_{2,2}=X_{2,1}T_{1,2}Y_{2,1}\tilde{T}_{1,2}$,
	which implies that $X_{2,1}T_{1,2}Y_{2,1}\tilde{T}_{1,2}\in\mbox{ran}\ \tau_{\tilde{T}_{2,2}}$. From $T_{1,1}T_{1,2}=T_{1,2}T_{2,2}$, we get
	\begin{equation*}
		\setlength\abovedisplayskip{3pt}
		\setlength\belowdisplayskip{3pt}
		\begin{aligned}
			\tilde{T}_{2,2}X_{2,1}T_{1,2}Y_{2,1}\tilde{T}_{1,2}
			&=X_{2,1}T_{1,1}T_{1,2}Y_{2,1}\tilde{T}_{1,2}\\
			&=X_{2,1}T_{1,2}T_{2,2}Y_{2,1}\tilde{T}_{1,2}\\
			&=X_{2,1}T_{1,2}Y_{2,1}\tilde{T}_{1,1}\tilde{T}_{1,2}\\
			&=X_{2,1}T_{1,2}Y_{2,1}\tilde{T}_{1,2}\tilde{T}_{2,2},
		\end{aligned}
	\end{equation*}
	which means that $X_{2,1}T_{1,2}Y_{2,1}\tilde{T}_{1,2}\in\ker\tau_{\tilde{T}_{2,2}}$. From Corollary \ref{corollary3.5}, we can obtain that $X_{2,1}T_{1,2}Y_{2,1}\tilde{T}_{1,2}=0$.
	Since the nonzero operators $T_{1,2},\tilde{T}_{1,2}$ both have dense range, $X_{2,1}T_{1,2}Y_{2,1}=0,$ which shows that at least one of $X_{2,1}$ and $Y_{2,1}$ is equal to 0.
	Without loss of generality, we assume that $X_{2,1}$ is zero. From $XT=\tilde{T}X$ and $TY=Y\tilde{T}$, and by a calculation, we obtain $\tilde{T}XY=XY\tilde{T}$.
	Therefore, by Proposition \ref{proposition3.6}, $XY$ is upper-triangular, which implies that $X_{2,2}Y_{2,1}=0$.
	Since $X$ has dense range, so does $X_{2,2}$, and then $Y_{2,1}=0$. Consequently, we are led to the conclusion that both $X$ and $Y$ are upper-triangular.
	
	Suppose that $X=\big((X_{i,j})\big)_{i,j=1}^{n}$ and $Y=\big((Y_{i,j})\big)_{i,j=1}^{n}$ are two quasi-affine intertwining operators such that $XT=\tilde{T}X$ and $TY=Y\tilde{T}$. Then it is clear that $(XY)\tilde{T}=\tilde{T}(XY)$ and $(YX)T=T(YX)$, which give that  $XY$ and $YX$ are both upper-triangular by Proposition \ref{proposition3.6}.
	We will then prove the following claim.
	\begin{claim}
		$X_{n,j}=0$ and $Y_{n,j}=0$ for $1\le j\le n-1$.
	\end{claim}
	Below, we will prove by contradiction that this claim, namely that none of the three cases listed below can be true.
	\begin{enumerate}
		\item $X_{n,j}=0$, $1\le j\le n-1$, but $Y_{n,j}\ne0$ for some $1\le j\le n-1$.
		
		\item $Y_{n,j}=0$, $1\le j\le n-1$, but $X_{n,j}\ne0$ for some $1\le j\le n-1$.
		
		\item $X_{n,j}\ne0$ for some $1\le j\le n-1$ and $Y_{n,k}\ne0$ for some $1\le k\le n-1$.
	\end{enumerate}
	Since $XY$ is upper-triangular and $X_{n,n}$ has dense range, then $Y_{n,j}=0$ for all $1\le j\le n-1$ if $X_{n,j}=0$ for each $1\le j\le n-1$. Similarly, the fact that $YX$ is upper-triangular and $Y_{n,n}$ has dense range indicates that the second case is also invalid. So, next, we just consider the last possibility. Pick $p,q$ to be the smallest indices for which $X_{n,p}\ne0$ and $Y_{n,q}\ne0$. Consequently, the relationships $XT=\tilde{T}X$ and $TY=Y\tilde{T}$ yield that
	\begin{equation*}
		\setlength\abovedisplayskip{3pt}
		\setlength\belowdisplayskip{3pt}
		X_{n,p}T_{p,p}=\tilde{T}_{n,n}X_{n,p},\ X_{n,p}T_{p,p+1}+X_{n,p+1}T_{p+1,p+1}=\tilde{T}_{n,n}X_{n,p+1},\ Y_{n,q}\tilde{T}_{q,q}=T_{n,n}Y_{n,q}.
	\end{equation*}
	Since $T_{i,i}T_{i,i+1}=T_{i,i+1}T_{i+1,i+1}$ for $1\le i\le n-1$, we have
	$$\begin{array}{lll}
		&&X_{n,p}T_{p,p+1}T_{p+1,p+2}\cdots T_{n-1,n}Y_{n,q}\tilde{T}_{q,q+1}\cdots\tilde{T}_{n-1,n}\\
		&=&(\tilde{T}_{n,n}X_{n,p+1}-X_{n,p+1}T_{p+1,p+1})T_{p+1,p+2}\cdots T_{n-1,n}Y_{n,q}\tilde{T}_{q,q+1}\cdots\tilde{T}_{n-1,n}\\
		&=&\tilde{T}_{n,n}X_{n,p+1}T_{p+1,p+2}\cdots\tilde{T}_{n-1,n}-X_{n,p+1}T_{p+1,p+2}\cdots\tilde{T}_{n-1,n}\tilde{T}_{n,n}\\
		&=&\tau_{\tilde{T}_{n,n}}(X_{n,p}T_{p,p+1}T_{p+1,p+2}\cdots T_{n-1,n}Y_{n,q}\tilde{T}_{q,q+1}\cdots\tilde{T}_{n-1,n}).
	\end{array}$$
	Moreover,
	$$\begin{array}{lll}
		&&\tilde{T}_{n,n}X_{n,p}T_{p,p+1}T_{p+1,p+2}\cdots T_{n-1,n}Y_{n,q}\tilde{T}_{q,q+1}\cdots\tilde{T}_{n-1,n}\\
		&=&X_{n,p}T_{p,p}T_{p,p+1}T_{p+1,p+2}\cdots T_{n-1,n}Y_{n,q}\tilde{T}_{q,q+1}\cdots\tilde{T}_{n-1,n}\\
		&=&X_{n,p}T_{p,p+1}T_{p+1,p+2}\cdots T_{n-1,n}T_{n,n}Y_{n,q}\tilde{T}_{q,q+1}\cdots\tilde{T}_{n-1,n}\\
		&=&X_{n,p}T_{p,p+1}T_{p+1,p+2}\cdots T_{n-1,n}Y_{n,q}\tilde{T}_{q,q}\tilde{T}_{q,q+1}\cdots\tilde{T}_{n-1,n}\\
		&=&X_{n,p}T_{p,p+1}T_{p+1,p+2}\cdots T_{n-1,n}Y_{n,q}\tilde{T}_{q,q+1}\cdots\tilde{T}_{n-1,n}\tilde{T}_{n,n},
	\end{array}$$
	which implies that $X_{n,p}T_{p,p+1}T_{p+1,p+2}\cdots T_{n-1,n}Y_{n,q}\tilde{T}_{q,q+1}\cdots\tilde{T}_{n-1,n}\in\ker\tau_{\tilde{T}_{n,n}}$. Using Corollary \ref{corollary3.5}, we can now derive that $X_{n,p}T_{p,p+1}T_{p+1,p+2}\cdots T_{n-1,n}Y_{n,q}\tilde{T}_{q,q+1}\cdots\tilde{T}_{n-1,n}=0$.
	Note that nonzero operators $T_{p,p+1},\cdots,T_{n-1,n}$,  $\tilde{T}_{q,q+1},\cdots$, $\tilde{T}_{n-1,n}$ and $Y_{n,q}$ have dense range and consequently $X_{n,p}=0$, which leads to a contradiction.
	
	Let us write the two operators $T,\tilde{T}$ in the form of $2\times 2$ block matrices:
	\begin{equation*}
		\setlength\abovedisplayskip{3pt}
		\setlength\belowdisplayskip{3pt}
		T=\left(\begin{smallmatrix}T_{n-1}&T_{(n-1)\times 1}\\0&T_{n,n}\end{smallmatrix}\right),\
		\tilde{T}=\left(\begin{smallmatrix}\tilde{T}_{n-1}&\tilde{T}_{(n-1)\times 1}\\0&\tilde{T}_{n,n}\end{smallmatrix}\right).
	\end{equation*}
	Correspondingly, $X,Y$ can be written as follows:
	\begin{equation*}
		\setlength\abovedisplayskip{3pt}
		\setlength\belowdisplayskip{3pt}
		X=\left(\begin{smallmatrix}X_{n-1}&X_{(n-1)\times 1}\\0&X_{n,n}\end{smallmatrix}\right),\
		Y=\left(\begin{smallmatrix}Y_{n-1}&Y_{(n-1)\times 1}\\0&Y_{n,n}\end{smallmatrix}\right),
	\end{equation*}
	where $X_{n-1}=\big((X_{i,j})\big)_{i,j=1}^{n-1}$, $Y_{n-1}=\big((Y_{i,j})\big)_{i,j=1}^{n-1}$ and $X_{(n-1)\times 1}=(X_{1,n},\cdots,X_{n-1,n})\textsuperscript{T}$, $Y_{(n-1)\times 1}=(Y_{1,n},\cdots,Y_{n-1,n})\textsuperscript{T}$.
	Following the relation $XT=\tilde{T}X$ and $TY=Y\tilde{T}$, we get
	\begin{equation*}
		\setlength\abovedisplayskip{3pt}
		\setlength\belowdisplayskip{3pt}
		\left(\begin{smallmatrix}X_{n-1}&X_{(n-1)\times 1}\\0&X_{n,n}\end{smallmatrix}\right)
		\left(\begin{smallmatrix}T_{n-1}&T_{(n-1)\times 1}\\0&T_{n,n}\end{smallmatrix}\right)
		=\left(\begin{smallmatrix}\tilde{T}_{n-1}&\tilde{T}_{(n-1)\times 1}\\0&\tilde{T}_{n,n}\end{smallmatrix}\right)
		\left(\begin{smallmatrix}X_{n-1}&X_{(n-1)\times 1}\\0&X_{n,n}\end{smallmatrix}\right)
	\end{equation*}
	and
	\begin{equation*}
		\setlength\abovedisplayskip{3pt}
		\setlength\belowdisplayskip{3pt}
		\left(\begin{smallmatrix}Y_{n-1}&Y_{(n-1)\times 1}\\0&Y_{n,n}\end{smallmatrix}\right)
		\left(\begin{smallmatrix}\tilde{T}_{n-1}&\tilde{T}_{(n-1)\times 1}\\0&\tilde{T}_{n,n}\end{smallmatrix}\right)
		=
		\left(\begin{smallmatrix}T_{n-1}&T_{(n-1)\times 1}\\0&T_{n,n}\end{smallmatrix}\right)
		\left(\begin{smallmatrix}Y_{n-1}&Y_{(n-1)\times 1}\\0&Y_{n,n}\end{smallmatrix}\right).
	\end{equation*}
	Then we know that $X_{n-1}T_{n-1}=\tilde{T}_{n-1}X_{n-1}$ and $Y_{n-1}\tilde{T}_{n-1}=T_{n-1}Y_{n-1}$.
	By an analogous proof, we can obtain that $X_{i,j}=0$ and $Y_{i,j}=0$ for all $1\le j<i\le n-1$, which shows that $X$ and $Y$ are both upper-triangular.
\end{proof}
Moreover, this Proposition \ref{proposition3.7} is also valid if the intertwining operator between any two operators in $\mathcal{F}_{n}(\mathcal{A})$ is invertible.

\end{document}